\documentclass[12pt]{amsart}
\usepackage{amsmath,amssymb,amsthm}
\usepackage{verbatim}
\usepackage{libertine}
\usepackage{enumitem}
\usepackage{bm} 
\usepackage[libertine]{newtxmath}
\usepackage[colorlinks, urlcolor=blue,  citecolor=blue]{hyperref}
\usepackage[all]{xy}%
\usepackage[utf8]{inputenc}
\title[Strong approximation for affine quadrics]{
	Effective equidistribution, arithmetic purity of strong approximation, and geometric sieve for affine quadrics}
\author{Yang Cao \and Zhizhong Huang \and Runlin Zhang}
\date{\today}

\address{School of Mathematics, Shandong University, 27 Shanda South Road, 250100 Jinan, China}
\email{yangcao1988@gmail.com}

\address{State Key Laboratory of Mathematical Sciences, Academy of Mathematics and Systems Science, Chinese Academy of Sciences, Beijing 100190, China}
\email{zhizhong.huang@yahoo.com}

\address{College of Mathematics and Statistics, Center of Mathematics, Chongqing University, 401331, Chongqing,  China}
\email{zhangrunlin@outlook.com}

\date{\today}

\def\setminus{\mathchoice
	{\mathbin{\vrule height .92ex width 1.81ex depth -.58ex}}
	{\mathbin{\vrule height .92ex width 1.81ex depth -.58ex}}
	{\mathbin{\vrule height .65ex width 1.00ex depth -.43ex}}
	{\mathbin{\vrule height .50ex width 0.770ex depth -.34ex}}
}

\newcommand{\APHL}{\textbf{(APHL)}}

\newcommand{\APSA}{\textbf{(APSA)}}

\newcommand{\SL}{\operatorname{SL}}

\newcommand{\BA}{{\mathbb {A}}}
\newcommand{\BB}{{\mathbb {B}}}

\newcommand{\BF}{{\mathbb {F}}}

\newcommand{\BN}{{\mathbb {N}}}

\newcommand{\BP}{{\mathbb {P}}}
\newcommand{\BQ}{{\mathbb {Q}}}
\newcommand{\BR}{{\mathbb {R}}}

\newcommand{\BT}{{\mathbb {T}}}

\newcommand{\BV}{{\mathbb {V}}}

\newcommand{\BZ}{{\mathbb {Z}}}
\newcommand{\CA}{{\mathcal {A}}}
\newcommand{\CB}{{\mathcal {B}}}

\newcommand{\CD}{{\mathcal {D}}}
\newcommand{\CE}{{\mathcal {E}}}
\newcommand{\CF}{{\mathcal {F}}}
\newcommand{\CG}{{\mathcal {G}}}

\newcommand{\CI}{{\mathcal {I}}}

\newcommand{\CN}{{\mathcal {N}}}
\newcommand{\CO}{{\mathcal {O}}}
\newcommand{\CP}{{\mathcal {P}}}

\newcommand{\CS}{{\mathcal {S}}}
\newcommand{\CT}{{\mathcal {T}}}
\newcommand{\CU}{{\mathcal {U}}}
\newcommand{\CV}{{\mathcal {V}}}

\newcommand{\CX}{{\mathcal {X}}}
\newcommand{\CY}{{\mathcal {Y}}}
\newcommand{\CZ}{{\mathcal {Z}}}
\newcommand{\RA}{{\mathbf {A}}}

\newcommand{\GL}{{\mathrm{GL}}}

\newcommand{\ord}{{\mathrm{ord}}}

\renewcommand{\mod}{\ \mathrm{mod}\ }
\renewcommand{\Re}{{\mathrm{Re}}}

\font\cyr=wncyr10
\newcommand{\Sha}{\hbox{\cyr X}}

\newcommand{\norm}[1]{\|{#1}\|}

\newcommand{\bs}{\backslash}

\newcommand{\sbt}{\subset}

\newcommand{\pr}{\operatorname{pr}}

\newcommand{\blambda}{\boldsymbol{\lambda}}
\newcommand{\bxi}{\boldsymbol{\xi}}
\newcommand{\bx}{\boldsymbol{x}}
\newcommand{\by}{\boldsymbol{y}}
\newcommand{\bz}{\boldsymbol{z}}

\newcommand{\bY}{\boldsymbol{Y}}

\newcommand{\bmone}{\bm{\mathrm{1}}}

\newcommand{\diff}{\mathrm{d}}  

\newcommand{\frakh}{\mathfrak{h}}
\newcommand{\fraks}{\mathfrak{s}}
\newcommand{\frakg}{\mathfrak{g}}

\newcommand{\rmm}{\mathrm{m}}

\newcommand{\calN}{\mathcal{N}}
\newcommand{\calO}{\mathcal{O}}

\newcommand{\calS}{\mathcal{S}}

\newcommand{\scrA}{\mathscr{A}}

\newcommand{\scrC}{\mathscr{C}}
\newcommand{\scrD}{\mathscr{D}}

\newcommand{\scrH}{\mathscr{H}}

\newcommand{\scrX}{\mathscr{X}}
\newcommand{\scrY}{\mathscr{Y}}
\newcommand{\scrZ}{\mathscr{Z}}
\newcommand{\scrO}{\mathscr{O}}

\newcommand{\Z}{\mathbb{Z}}

\newcommand{\R}{\mathbb{R}}
\newcommand{\C}{\mathbb{C}}

\newcommand{\ep}{\varepsilon}

\newcommand{\embed}{\hookrightarrow}

\newcommand{\midd}{\;\middle\vert\;}

\newcommand{\normm}[1]{\left\lvert#1\right\rvert}

\newtheorem{thm}{Theorem}[section]

\newtheorem{defi}[thm]{Definition}

\newtheorem{lem}[thm]{Lemma}

\theoremstyle{plain}

\newtheorem{theorem}{Theorem}[section]

\newtheorem{lemma}[theorem]{Lemma}
\newtheorem{corollary}[theorem]{Corollary}
\newtheorem{proposition}[theorem]{Proposition}

\theoremstyle{definition}

\newtheorem{example}[theorem]{Example}

\newtheorem*{remark*}{Remark}
\newtheorem*{remarks*}{Remarks}

\newtheorem{setting}[theorem]{Setting}

\usepackage{geometry}
\begin{document}
	\begin{abstract}
		Let $k$ be a number field. Let $q(x_1,\cdots,x_n)$ be a non-degenerate integral quadratic form in $n\geq 3$ variables with coefficients in $k$ and $m\in k^\times$. Let $X$ be the affine quadric defined by $q=m$ in $\BA^n_k$.
		Based on results on effective equidistribution of $S$-integral points in symmetric spaces, we establish the following: \begin{enumerate}
			\item The arithmetic purity of strong approximation off any single place of $k$  for $X$;
			\item The geometric sieve for $p_0$-integral points on $X$ when $k=\BQ$ and $p_0$ is a prime number.
		\end{enumerate}
	\end{abstract}
	\maketitle
		 \tableofcontents
		 
		 \section{Introduction}
	Let $X$ be a geometrically integral variety over  a number field $k$. It is a fundamental problem to study the density of rational points in $X$. A quanlitative way is to embed $X(k)$ into a suitable adelic space $X(\RA_{k}^S)$, $S$ being a finite set of places of $k$, and ask about the density inside with respect to the corresponding adelic topology. We say that $X(k)$ satisfies \emph{strong approximation off $S$} if $X(k)$ is dense in $X(\RA_{k}^S)$. It is a natural question (first proposed by Wittenberg \cite[Question 2.11]{Wit16}, see also \cite[Question 1.2]{CLX19}) that whether the strong approximation property is inherited upon removing any closed subset of codimension at least two. This motivates the following:
	\begin{defi}[cf. \cite{CLX19} Definition 2.1, \cite{Cao-Huang1} Definition 1.3]
		We say that $X$ satisfies the \emph{arithmetic purity of strong approximation off $S$}, abbreviated as \emph{\APSA\ off $S$}, if $U\subset X$ is an open subset such that the codimension of $X\setminus U$ is at least two, then $U(k)$ is dense in $U(\RA_{k}^S)$.
	\end{defi}
	 
It is well-known (due to Kneser, Platonov et. al, cf. \cite[Theorem 1.1]{Cao-Huang1}) that any non-trivial semi-simple simply connected algebraic group $G$ satisfies strong approximation off any place $v$, provided $G(k_v)$ has not compact factors. Supported by the fact that removing codimension two closed subsets does not create new topological obstruction, we expect that any homogeneous space under such a group $G$ should satisfy \APSA\ off $v$ (with Brauer--Manin obstruction whenever the stabilizer is not simply connected). 
	We refer to \cite[\S1.1]{Cao-Huang1} for a detailed account of known results towards Wittenberg's question. The main difficulty is that the adelic topology of a proper open subset differs from that of the ambient space, since infinitely many local conditions arise. 
	
	
	Our first main result fully confirms the \APSA\ property for affine quadrics. Let $q=q(x_1,\cdots,x_n)$ be an non-degenerate quadratic form in $n$-variables with coefficients in $k$, and $m\in k^\times$. Let $X$ be the affine variety \begin{equation}\label{eq:affquadric}
	    (q=m)\subset \BA_k^{n},
	\end{equation} which we call an \emph{affine quadric}. We always assume that $X(\RA_k)\neq \varnothing$, which is equivalent to, by the Hasse--Minkowski theorem, $X(k)\neq\varnothing$. 
	It is classically known (due to Eicher, Kneser et. al) that when $n\geq 3$, $X$ satisfies strong approximation (with Brauer--Manin obstruction if $n=3$) off any single place $v$ of $k$, provided that $X(k_v)$ is not compact. Note that this condition is equivalent to $q$ being isotropic over $k_v$. 
	\begin{thm}\label{thm:affinequadricsAPSA}
		Assume that $n\geq 4$. The affine quadric $X$ above satisfies \APSA\ off any single place $v$ of $k$, provided that $q$ is isotropic over $k_v$. 
	\end{thm}
		Theorem \ref{thm:affinequadricsAPSA} generalizes \cite[Theorem 1.6]{Cao-Huang1}, where the \APSA\ property is proven only off all the archimedean places. Combined with arguments in \cite{CLX19,Cao-Huang1}, we can also  confirm \APSA\ off $v$ with Brauer--Manin obstruction for affine quadrics defined by ternary quadratic forms.

		A closely related quantitative aspect is the so-called Hardy--Littlewood property, first introduced in work of Borovoi--Rudnick \cite{Browning-Gorodnik}. Let $X$ be a geometrically integral quasi-affine variety equipped with a gauge form, to which we associate a Tamagawa measure $\rmm_X$ on $X(\RA_{k})$ (cf. \cite[\S1]{Borovoi-Rudnick}, for instance, homogeneous spaces under connected reductive groups with unimodular stabilizers). Let $v$ be a place of $k$ and let $\operatorname{Ht}_v:X(k_v)\to \BR_{\geq 0}$ be a fixed $v$-adic height function (cf. \eqref{eq:height}). Let $\boldsymbol{\delta}:X(\RA_{k})\to\BR_{\geq 0}$ be a locally constant function which is not identically zero. Recall that  (cf. \cite[Definitions 2.2 \& 2.3]{Borovoi-Rudnick}) $X$ is \emph{Hardy--Littlewood off $v$} (with respect to $\boldsymbol{\delta},\rmm_X$) if $X(k_v)$ is not compact and for any compact set $B_v\subset X(\RA_{k}^{v})$, on writing $\BB_v(N)=B_v\times\{\bx\in X(k_v):\operatorname{Ht}_v(\bx)\leq N\}\subset X(\RA_{k})$\footnote{Note that since $\BB_v(N)$ has compact closure in $X(\RA_{k})$, and $X(k)$ is discrete in $X(\RA_{k})$, the cardinality $\#\left(X(k)\cap \BB_v(N)\right)$ is finite.}, $$\#\left(X(k)\cap \BB_v(N)\right)\sim \int_{\BB_v(N)} \boldsymbol{\delta}\operatorname{d}\rmm_X,\quad \text{as }N\to\infty.$$  The Hardy--Littlewood property is intimately tied to many analytic and dynamical methods that can be used to approach strong approximation (cf. \cite[Introduction]{Borovoi-Rudnick}). The following is the direct analogue of the \APSA\ property.
		
		\begin{defi}[cf. \cite{Cao-Huang2} Definition 1.2]
			We say that $X$ satisfies the \emph{arithmetic purity of the Hardy--Littlewood property} (\APHL\ for short) \emph{off} $v$, if any open subset $U\subset X$, whose complement has codimension at least two, is Hardy--Littlewood off $v$ (with respect to $\boldsymbol{\delta}|_{U(\RA_{k})},\rmm_X|_{U}$).
		\end{defi}
Note that the restriction $\rmm_X|_{U}$ is well-defined by (a slight generisation of) \cite[Proposition 2.1]{Cao-Huang2}. 

The first cases where the \APHL\ property are established are affine quadrics over $\BQ$ with $v=\BR$ \cite[Theorems 1.4 \& 1.5]{Cao-Huang2}. Let \begin{equation*}
    X=(q=m)\subset\BA_{\BQ}^n
\end{equation*} be an affine quadric over $\BQ$ with $X(\RA_{\BQ})\neq\varnothing$. Recall that (cf. \cite[\S6.4]{Borovoi-Rudnick}) we can associate a locally constant function $\boldsymbol{\delta}$ on $X(\RA_{\BQ})$ which is either identically one if $n\geq 4$, or takes values in $\{0,2\}$ if $n=3$ and $-m\det(q)\neq \square$. As is shown in work of Wei--Xu \cite{Wei-Xu}, in the latter case $\boldsymbol{\delta}$ serves as an indicator function detecting the Brauer--Manin locus of $X(\RA_{\BQ})$. Our second main result concerning \APHL\ off any non-archimedean place completes the remaining cases.

\begin{thm}
	Assume that $q$ is isotropic over $\BQ_{v}$, $v$ being a non-archimedean place. Assume moreover 
	\begin{itemize}
		\item either $n\geq 4$;
		\item or $n=3$ and $q$ is anisotropic over $\BQ$ \& $-m\det(q)\neq \square$.
	\end{itemize}
	Then $X$ satisfies \APHL\ off $v$ with respect to $\boldsymbol{\delta}$.
\end{thm}
		
		See \cite[\S IV.2 Theorem 6]{Serre} and \cite[Chapter VI \S63]{OMeara} for necessary and sufficient conditions for the quadratic form $q$ to be isotropic over $\BQ_{v}$, $v$ being non-archimedean. In particular, when $n\geq 5$ this condition becomes vacuous.

\subsection*{Methods of proof and organisation of the article}
Our approach relies on a number of upgraded  arguments established in previous work of the first and second authors \cite{Cao-Huang1,Cao-Huang2}, with the crucial dynamical input from equidistribution of $S$-integral point with effective power-saving error terms (see \S\ref{se:effective}), extending work of Benoist--Oh \cite{Benoist-Oh} and Browning--Gorodnik \cite{G-N}. 

The application of effective equidistribution is twofold. Firstly, it allows to execute the \emph{affine linear sieve} (first developed by Sarnak et. al) so as to produce ``almost-prime'' $S$-integral points. Our way of running this sieve (see \S\ref{se:affinelinearsieve}) closely ties to the Hardy--Littlewood property, which (we believe) appears more conceptural than previous work. Combining this with the sophisticated fibration method developed in \cite{Cao-Huang1} leads to the \APSA\ property for three-dimensional groups and eventually yields the same for affine quadrics (see \S\ref{se:APSA}). This sieve being executed over the number field $k$ without recurring to the Weil restriction (cf. \cite[Proof of Corollary 3.3]{Cao-Huang1}), this is the key reason why we can achieve \APSA\ off any single place. Secondly, in showing the \APHL\ property, which is the content of \S\ref{se:geomsieve}, it is useful in handling a certain range of ``bad primes''.

It has been realised in \cite{Browning-HB,Cao-Huang2} that in course of establishing the \APHL\ property, the treatment of remainder terms is essentially built on the \emph{geometric sieve} (first introduced by Ekedahl \cite{Ekedahl} for affine spaces). We prove a $p$-integral version of this in \S\ref{se:geomsieve}, following the strategy in \cite{Cao-Huang2}.

As slight generalization of work of Borovoi--Rudnick \cite{Borovoi-Rudnick}, in Appendix \ref{se:Siegel-Weil} we formulate a version of Siegel--Weil type formulas. They refer to summing over the leading constant showing up when one counts points in each $G$-orbit, which was usually imprecise in dynamical methods, gives rise to the Hardy--Littlewood density of the homogeneous space $H\bs G$.  

\subsection*{Notation}
Throughout this article, $k$ is a number field and $\calO$ denotes the ring of integers. We write $\Omega_k$ for the set of all places and $\infty_k$ for the set of infinite places. For $S$ any finite set of places of $k$, let $\mathcal{O}_{S}$ denote the ring of $(S\cup\infty_k)$-integers. For $v\in\Omega_{k}$, $k_v$ denotes the $v$-adic local field, and if $v$ is finite, $\CO_{v}$ denotes the ring of integers of $k_v$ and $\mathfrak{m}_{v}$ the maximal ideal.
 For a prime ideal $\mathfrak{p}$ of $\CO$, we write $\BF_{\mathfrak{p}}$ for its residue field. For $\mathfrak{b}$ an ideal of $\CO_{S}$, we write $\mathbf{N}(\mathfrak{b}):=\#\left(\CO_{S}/\mathfrak{b}\right)$ for the norm of $\mathfrak{b}$. 
 
A $k$-variety is a separate scheme of finite type over $k$. For any $k$-variety $X$, we write $X_S=\prod_{v \in S} X(k_v)$, and $X(\RA_k^S)$ be the adelic space off $S$. If $\CX$ is an $\CO_{S}$-scheme, 
For each place $v\not\in (S\cup\infty_k)$ and integer $m$ the reduction modulo $\mathfrak{m}_{v}^m$ maps is denoted by \begin{equation}\label{eq:modvm}
	\operatorname{Mod}_{v,m}:\CX(\CO_v)\longrightarrow \CX(\CO_v/\mathfrak{m}_{v}^m).
\end{equation}
	
	\section{Effective counting of $S$-integral points}\label{se:effective}
	
	In this section we fix the following data:
	\begin{setting}\label{setting}\hfill
		\begin{itemize}
			\item $G$ is a connected, simply connected, almost $k$-simple linear algebraic group over $k$.
			\item $V$ is an $n$-dimensional $k$-vector space together with a fixed isomorphism $V \cong k^n$.
			\item $\iota: G \to GL(V)$ is a faithful representation over $k$.
			\item $S$ is a finite set of places of $k$ containing all infinite places $v$ over which $G(k_v)$ has no compact factors (which may not necessarily exist), and write $S = S_{\infty} \sqcup S_f$ as the disjoint union of infinite and finite places. 
			\item $H \subset G$ is a symmetric subgroup whose identity component has no nontrivial $k$-characters.
			\item $\Gamma := \{ g \in G_S \mid \iota(g) \calO_{S}^n = \calO_{S}^n \}$, $Y_S:= \Gamma \bs G_S$, $Z_S:= \Gamma\bs \Gamma H_S$.
			\item $z_0 \in V$ whose stabilizer in $G$ is $H$, and $X:= z_0.G$.
			\item Fix Haar measures \( \rmm_{G} \) on \( G \), \( \rmm_{H} \) on \( H \), and \( \rmm_{X} \) on \( X \) that are compatible. 
		\end{itemize}
	\end{setting}
	We shall normalise the measures $\rmm_{G},\rmm_{H}$ such that \begin{equation}\label{eq:normalisation}
		|\rmm_{Y_S}| :=\rmm_{G_S}\left(Y_S\right)=1,\quad  |\rmm_{Z_S}|:=\rmm_{H_S}\left(Z_S\right) = 1. 
	\end{equation} 
	
	
		\subsection{Main theorem on effective equidistribution of $S$-integral points} 	Our main theorem below is an $S$-adic generalization of work of Browning-Gorodnik \cite[Theorem 2.1]{Browning-Gorodnik} on the one hand. It also upgrades work of Benoist-Oh \cite[Theorem 1.3]{Benoist-Oh} with congruence conditions on the other hand. To state it, we adopt the following notation:
	
	\begin{itemize}
		\item  If $\mathfrak{l}$ is an ideal of $\mathcal{O}_{k,S}$, we let 
	\begin{equation}\label{eq:gammal}
			\Gamma(\mathfrak{l}):= \ker\left(
		\Gamma  \to \mathrm{GL}_n (\mathcal{O}_S )
		\to \mathrm{GL}_n (\mathcal{O}_S/\mathfrak{l} \calO_S)
		\right)
	\end{equation}
		and  \[
		Y_{S}(\mathfrak{l}) := \Gamma(\mathfrak{l})\bs G_S,\quad
		Z_S(\mathfrak{l}):= \Gamma(\mathfrak{l})\bs \Gamma(\mathfrak{l})H_S.
		\]
		\item For $\bx= (x_1,\cdots,x_n)\in V$, let 
		\begin{equation}\label{eq:height}
		\operatorname{Ht}_{v}(\bx):=
		\begin{cases}
			\max\left\{ \normm{x_1}_{v},\cdots \normm{x_n}_{v} \right\}, & v \in S_f; \\
			(\sum_{i=1}^n |x_i|_{v}^2)^{1/2},& v \in S_{\infty}.
		\end{cases}\;
		\end{equation}
		\item We fix a partition $S=S_1\sqcup S_2$ where $S_1\subset S$ consists only of places over which $G$ has no compact factors. Let $\mathfrak{D}_{S_2} \subset X_{S_2}$ be a region with piecewise smooth boundary and compact closure. Also let $$B_{S_1}(N):= \left\{
		\bx \in X_{S_1} \midd \prod_{v\in S_1}\operatorname{Ht}_v(\bx) \leq N
		\right\}$$ for $N\in \Z^+$ and write for simplicity \begin{equation}\label{eq:BSN}
			\BB_S(N):=B_{S_1}(N) \times \mathfrak{D}_{S_2}.
		\end{equation}
	\end{itemize}

	\begin{thm}\label{thm:effectivemain}
		Assume that $S_1\neq\varnothing$. Then there exist constants \( \kappa > 1,\;0<\delta<1  \) depending on the data in Setting \ref{setting} such that uniformly for every ideal \( \mathfrak{l}\) of \(\calO_{S} \), \( \gamma \in \Gamma \), and \( N \in \mathbb{Z}^+ \),
		\[
		\left\vert z_0 .\gamma \Gamma(\mathfrak{l}) \cap 	\BB_S(N) \right\vert  -
		\frac{  \normm{\rmm_{Z_S(\mathfrak{l})}} 
		}{   \normm{\rmm_{Y_S(\mathfrak{l})} } 
		} \rmm_{X_{S}}(	\BB_S(N) ) = O \left(\mathbf{N}(\mathfrak{l})^{\kappa} \cdot \rmm_{X_{S}}(\BB_{S}(N)) ^{1-\delta}\right),
		\] where the implied constant may also depend on $\mathfrak{D}_{S_2}$.
	\end{thm}
	
	Note that the Haar measures $\rmm_{G},\rmm_{H}$ induce measures on the quotients: \( \rmm_{Y_S(\mathfrak{l})} \) on \( Y_S(\mathfrak{l}) \), \( \rmm_{Z_S(\mathfrak{l})} \) on \( Z_S(\mathfrak{l}) \) that are locally the same as \( \rmm_{G_S} \) or \( \rmm_{H_S} \). Under the normalisation \eqref{eq:normalisation},
	\[
	\normm{\rmm_{Y_S(\mathfrak{l})} } = \normm{\Gamma(\mathfrak{l})\bs \Gamma},
	\quad
	\normm{\rmm_{Z_S(\mathfrak{l})}}  = |\Gamma(\mathfrak{l}) \cap H_S \bs \Gamma \cap H_S|.
	\]
	
	There is also a small generalization that we now turn to. 
	\begin{thm}\label{thm:GtimesG}
		Assume that $G=G_1 \times G_1$, where $G_1$ is a connected, simply connected, almost $k$-simple linear algebraic group, and that $H\subset G$ is the diagonal embedding of $G_1$. Keeping all the remaining assumptions, the conclusion of Theorem \ref{thm:effectivemain} holds.
	\end{thm}

	In the remaining part of this section we outline the proof of Theorems \ref{thm:effectivemain} and \ref{thm:GtimesG}. The crucial input for 
	\cite[Theorem 2.1]{Browning-Gorodnik} was an effective decay of matrix coefficients from \cite{Kleinbock_Margulis_logarithm_laws}, which needs to be replaced by an S-adic variant (Theorem \ref{theorem_B} below). This follows from the proof of \cite[Theorem 3.22]{GorOhMau08}. After that, one essentially repeats the proof in \cite{Benoist-Oh}. To avoid redundancy and make the proof transparent at the same time, we choose to state rigorously how several intermediate steps in \cite{Benoist-Oh} are modified, give a sketch of proof when we feel necessary. Then we explain the (standard) procedure of deducing Theorems \ref{thm:effectivemain} and \ref{thm:GtimesG} from them. 
		In \S\ref{se:affquavol} and \S\ref{se:volasy} we briefly discuss the computation of volume asymptotics for  affine quadrics and semisimple groups, following \cite{Chamber-Loir_Tschinkel_2010_Igusa_integral}.
		
	\subsection{Intermediate lemmas}\label{sec:intermediatelemma}
    The proof of Theorem \ref{thm:effectivemain} is divided into several parts just as in \cite{Benoist-Oh}:
    \begin{itemize}
        \item effective well-roundedness: 
        \ref{lemma_A};
		\item   effective mixing: 	$\ref{theorem_B}
			\implies 
			\ref{theorem_C}$;
		\item effective equidistribution: $\ref{theorem_C}, 
			\,\ref{lemma_D}, 
			\,\ref{lemma_E}, 
			\,\ref{lemma_F}, 
			\, \ref{lemma_G} 
			\implies 
			\ref{theorem_H}$;
		\item Wrap-up: $\ref{lemma_A},\,\ref{lemma_E},\,
			\ref{theorem_H} \implies
			\ref{thm:effectivemain}$.
    \end{itemize}
    In this subsection, we only state the modified statements for the first three parts. The last wrap-up step is carried out in \S\ref{se:pfeffectivemain}.

	We start with the
	effective well-roundedness (see \cite[Definition 13.1]{Benoist-Oh}) of the family $(\BB_S(N))$.
	For $\varepsilon>0$, let $$\mathcal{N}^{\infty}_{\ep} := \{ g  \in G_S : \|\iota(g) - \text{id}\|_{v} \leq \varepsilon, \; \forall \,v \in \mathcal{S}_{\infty} \},$$ and 
	$$B^+_{N, \ep} := \BB_S(N) \cdot \calN^{\infty}_{\ep},\quad
			B^-_{N, \ep} := \bigcap_{b\in \calN^{\infty}_{\ep}} \BB_S(N)\cdot b.$$
	The following is \cite[Proposition 14.2]{Benoist-Oh}.
	\begin{lem}\label{lemma_A}
	The following hold for every \( N \in \mathbb{Z}^+ \).
		\begin{enumerate}
			\item There exists an open compact subgroup \( U_f^0 \subset G_{S_f} \) such that \( \BB_S(N) \) is \(  U_f^0  \)-invariant;
			\item There exist constants \( \kappa_0, C_0 > 1 \) such that for all \( \varepsilon  \in (0, 1) \),
			\[
			\rmm_{X_S}({B^+_{N, \ep}} \setminus B^-_{N,\ep}) \leq C_0  \ep^{\kappa_0} \rmm_{X_S}({B_{N}}).
			\]
			\item For all \( d  \in (0, +\infty)\), there exists \( \delta_0\in (0, 1) \) such that for  all \( \varepsilon \in (0, 1) \),
			\[
			\int_{B^+_{N,\ep}} \Big (  {\prod_{v \in S}\mathrm{Ht}_{v}}(z) \Big)^{-d}\, \diff\rmm_{X_S}(z) 
			\leq C_0  \rmm_{X_S}{(\BB_S(N))}^{1 - \delta_0}.
			\]
		\end{enumerate}
	\end{lem}

	\begin{thm}\label{theorem_B}
		For every open compact subgroup \( U_f \leq G_S \), there exist \(C_1= C_1({U_f})>1 \) and \(m_1=m_1(U_f) > 0 \) such that for every \( U_f \)-invariant smooth vectors \( x,y \in L_0^2(G(k) \bs G(\RA_k)) \) (the Hilbert space of $L^2$-functions which integrate to zero against the invariant measure) and for every $g \in G_S$,
		\[
		\left| \langle g x, y \rangle_{G(k)\bs G(\RA_k)} \right| \leq C_1 \widetilde{\xi}_{S}(g) \norm{D^{m_1} x}_{L^2}  \norm{D^{m_1} y}_{L^2},  
		\]
		where $\widetilde{\xi}_S:G(\RA_k)\to (0,1]$ is defined in \cite[Definition 3.20]{GorOhMau08} (with $R$ replaced by $S$), and all integrals here are with respect to the \( G(\RA_k) \)-invariant probability measure on \( G(k)\bs G(\RA_k) \).
	\end{thm}
	As is mentioned, this does not formally follow from \cite[Theorem 3.22]{GorOhMau08}, but rather from the proof of it. We refer to the proof of \cite[Theorem 10.2]{Benoist-Oh} for useful properties of $\widetilde{\xi}_S$.  Repeating this proof with Theorem \ref{theorem_B} one has
	\begin{thm}\label{theorem_C}
		There exists \( \delta_1 > 0 \) such that for every open compact subgroup \( U_f \subset G_{S_f} \), there exists \(C_2= C_2{(U_f)} > 1 \) such that for each ideal $\mathfrak{l}$ of \(  \calO_{S} \), every $U_f$-invariant
		$\psi_1, \psi_2 \in C_c^{\infty}(Y_S(\mathfrak{l}))$ and for every $g \in G_S$,
		\[
		\normm{ \langle g \psi_1,\psi_2 \rangle_{{Y_S(\mathfrak{l})}} - 
			\int_{Y_S(\mathfrak{l})}\psi_1 \, \mathrm{d} \rmm_{Y_S(\mathfrak{l})}^{\bmone} \int_{Y_S(\mathfrak{l})}\psi_2 \, \mathrm{d} \rmm_{Y_S(\mathfrak{l})}^{\bmone}
			} \leq C_2 \calS^{m_1}(\psi_1)\calS^{m_1}(\psi_2) \left(\prod_{v\in S}\mathrm{Ht}_v(z_0.g) \right)^{-\delta_1}, 
		\]
		where the implicit integrals in defining Sobolev norms  $\calS^{m_1}$ as well as in $\langle -,-\rangle$ are against the normalized probability measures.
		The sup-index $(-)^{\bmone}$ indicates the probability measure proportional to the given finite measure.
	\end{thm}
	
	For $\varepsilon>0$, define
	$$\mathcal{N}^{f}_{\ep} := \{ g \in G_S : \|\iota(g) - \text{id}\|_{v} \leq \varepsilon, \; \forall \,v \in \mathcal{S}_{f} \}.$$
	The following lemma is taken from \cite[Lemmas 12.3, 3.3]{Benoist-Oh}, which explains how \( H \) being symmetric helps. The construction of \( F_S \) is found in \cite[Lemma 3.3]{Benoist-Oh}.
	\begin{lem}\label{lemma_D}
		There exists a subset \( F_S \subset G_S \) such that
		\begin{enumerate}
			\item  $G_S = H_S \cdot F_S$;
			\item \( \exists\, \ell_1 \in \mathbb{Z}^+, \,  \varepsilon_1 >0 \) such that for all \( 0 < \ep, \ep' < \varepsilon_1\) and \(  g\in F_S \),
			\[
			H_S \cdot \calN^{\infty}_{\ep} \calN_{\ep'}^f \cdot g \subset H_S \cdot g \cdot  \calN^{\infty}_{l_1\ep} \calN_{l_1\ep'}^f.
			\]
		\end{enumerate}
	\end{lem}
	Next, we explain the effective thickening trick, following \cite[Section 11]{Benoist-Oh}.
	\begin{itemize}
		\item Choose $\fraks$, an $\mathrm{Ad}(H)$-invariant $k$-subspace of $\frakg$ complementary to $\frakh$:
		$\frakg = \frakh \oplus \fraks$.
		And we let $W_{v}:= \exp(\fraks_{k_{v}})$, $W_S:= \prod_{v \in S} W_{v} \subset G_S$.
		\item Choose
		\( 0 < \ep_2 < \ep_1 \) such that
		\begin{enumerate}
			\item  $H_{\ep_2} \times W_{\ep_2} \to H_{\ep_2} \cdot W_{\ep_2} \to \Gamma \bs \Gamma H_{\ep_2}W_{\ep_2}$ are homeomorphisms where $H_{\ep_2}:= H_S \cap \calN^{\infty}_{\ep_2}\calN^{f}_{\ep_2}$ and $W_{\ep_2}:= W_S \cap \calN^{\infty}_{\ep_2}\calN^{f}_{\ep_2}$;
			\item $\calN^{f}_{\ep_2}$ is contained in an open compact subgroup of $G_{S_f}$.
		\end{enumerate}
		\item Fix a measure \( \rmm_W \) supported on \( W_{\ep_2} \) such that
		\[
		\rmm_{G_S} \vert_{H_{\ep_2}W_{\ep_2} } \cong \rmm_{H_S} \vert_{H_{\ep_2}} \otimes \rmm_W.
		\]
		\item For every open compact subgroup $U_f$ of $G_{S_f}$, we fix another open compact subgroup $U_f' \subset U_f$ such that $H_{\ep_2}W_{\ep_2}$ is preserved by $U_f'$.
	\end{itemize}
	
	The following is \cite[Lemma 11.8]{Benoist-Oh}.
	\begin{lem}\label{lemma_E}
		There exist \( C_3 > 1 \) and \( \kappa_2 :\mathbb{Z}^+ \to (1, +\infty) \) such that for every open compact subgroup \( U_f \subset G_{S_f} \) and  \( 0 < \ep < \ep_2 \), there are smooth functions \( \beta_{\ep}: H_S \to \mathbb{R}_{\geq 0} \) and  \(\rho_{\ep}: W_S \to \mathbb{R}_{\geq 0} \)  such that for all \( m \in \mathbb{Z}^+ \),
		\begin{itemize}
			\item[1.] \( \text{supp} \, \beta_{\varepsilon} \subset H_{\ep}, \; \beta_{\varepsilon} \vert_{H_{\ep^2}} \geq 1,\; \calS^m(\beta_{\ep}) \leq C_3 \ep^{-\kappa_2(m)} \),
			\item[2.] $ \text{supp} \, \rho_{\varepsilon} \subset W_{\ep}, \;  \int_{W_{\ep_2}} \rho_{\ep}\, \diff\rmm_{W} =1
			$,
		\end{itemize}
		and if \( \alpha_{\ep} \) denotes the smooth function supported on \( H_{\ep} W_{\ep} \) defined by \( \alpha_{\ep}{(hw)} := \beta_{\varepsilon}(h) \rho_{\ep}(w) \), then
		\begin{itemize}
			\item[3.]  $\alpha_{\ep} \text{ is } U'_f \text{-invariant}$ and $\calS^m(\alpha_{\ep}) \leq C_3  \varepsilon^{-\kappa_2(m)}$.
		\end{itemize}
	\end{lem}
	Note that ``$\calS^m(\beta_{\ep}) \leq C_3 \ep^{-\kappa_2(m)}$'' is implicit from the proof of  \cite[Lemma 11.8]{Benoist-Oh}.
	The following is a variant of \cite[Proposition 11.7]{Benoist-Oh}.
	\begin{lem}\label{lemma_F}
		There exist \(C_4 > 1\) and \( \kappa_3 : \mathbb{Z}^+ \to (0,+\infty)\) such that for every open compact subgroup \( U_f \subset G_{S_f} \), \( 0 < \varepsilon < \varepsilon_2, m\in \mathbb{Z}^+ \) and  ideal $\mathfrak{l}$ of $\calO_{S}$, there exists
		\begin{itemize}
			\item[1.] \( \tau^{\mathfrak{l}}_{\ep}:  Y_S(\mathfrak{l}) \to [0,1], \, \text{supp}(\tau^{\mathfrak{l}}_{\ep}) \subset Z_{\ep}(\mathfrak{l}),\;
			\tau^{\mathfrak{l}}_{\ep}\vert_{Y_{4\ep}(\mathfrak{l})} =1 \)
		\end{itemize}
		such that if \( \varphi_{\varepsilon}^{\mathfrak{l}} \) is the function on \( Z_{\ep}(\mathfrak{l}) \) defined by
		\[
		\varphi_{\varepsilon}^{\mathfrak{l}} (x) := 
		\frac{ \normm{ \rmm_{Y_S(\mathfrak{l})}}
		}{  \normm{  \rmm_{Z_S(\mathfrak{l})} }
		} 
		\sum_{ (y,w) \in Z_{\ep}(\mathfrak{l}) \times W_{\ep},\; x=yw
		}
		\tau^{\mathfrak{l}}_{\ep}(y) \rho_{\ep}(w),
		\]
		then
		\begin{itemize}
			\item[2.]  $ \tau_{\varepsilon}^{\mathfrak{l}} \text{ is } U'_f \text{-invariant and} \;\; \calS^m(\varphi_{\varepsilon}^{\mathfrak{l}}) \leq C_{4}        \ep^{-\kappa_3(m)} \left| \rmm_{Y_S(\mathfrak{l})}\right| \normm{\rmm_{Z_{S}(\mathfrak{l})}}$.
		\end{itemize}
	\end{lem}
	
	In the lemma above, we have adopted the notation (for \( 0 < \varepsilon < \varepsilon_2 \))
	\begin{equation*}
		\begin{aligned}
			Y_{\ep}(\mathfrak{l}) := \left\{ 
			x \in Y_S(\mathfrak{l}) \midd g \mapsto yg \text{ is injective on } H_{\ep}W_{\ep} \right\},\;
			Z_{\ep}(\mathfrak{l}):= Z_S(\mathfrak{l}) \cap Y_{\ep}(\mathfrak{l}).
		\end{aligned}
	\end{equation*}
	Let $\pi^Y_{\mathfrak{l}}$ (resp. $\pi^Z_{\mathfrak{l}}$) denote the natural quotient map from $Y_S(\mathfrak{l}) $ to $Y_S$ (resp. $Z_S(\mathfrak{l}) $ to $Z_S$).
	Here is a slight variant of \cite[Lemma 11.6]{Benoist-Oh}.
	\begin{lem}\label{lemma_G}
		There exist $0<\ep_3<\ep_2$ and $C_5, \kappa_4 > 1$ such that for each $0<\ep<\ep_3$ and ideal $\mathfrak{l}$ of $\calO_{S}$,
		\[
		\rmm^{\bmone}_{Z_S(\mathfrak{l})}
		\left({Z_S(\mathfrak{l}) \setminus Z_{\ep}(\mathfrak{l})} \right) \leq C_5 \varepsilon^{\kappa_4}.
		\]
	\end{lem}
	\begin{proof}
		When $L=(1)$, \cite[Lemma 11.6]{Benoist-Oh} implies $\rmm^{\bmone}_{Z_S}( Z_S \setminus Z_{\ep}) \leq C_5 \ep^{\kappa_4}$ for some $C_5,\kappa_4>1$, $0<\ep_3<\ep_2$ and all $0<\ep<\ep_3$.
		Since $\pi^Z_{\mathfrak{l}}( Z_S(\mathfrak{l})\setminus Z_{\ep}(\mathfrak{l}) ) \subset Z_S \setminus Z_{\ep}$ by definition and $( \pi^Z_{\mathfrak{l}})_* \rmm^{\bmone}_{ Z_S(\mathfrak{l})} = \rmm^{\bmone}_{Z_S}$. We have
		\begin{equation*}
			\rmm^{\bmone}_{Z_S(\mathfrak{l})}\left({Z_S(\mathfrak{l}) \setminus Z_{\ep}(\mathfrak{l})} \right)
			\leq 
			\rmm^{\bmone}_{Z_S(\mathfrak{l})}\left(
			(\pi^Z_{\mathfrak{l}})^{-1}
			({Z_S \setminus Z_{\ep}}) \right) =\rmm^{\bmone}_{Z_S}(Z_S \setminus Z_{\ep}) \leq C_5 \ep^{\kappa_4}. \qedhere
		\end{equation*}
	\end{proof}
	
	Here is the analogue of \cite[Theorem 12.5]{Benoist-Oh}.
	\begin{thm}\label{theorem_H}
		There exist \( m_2 \in \mathbb{Z}^+, \, \kappa_5 > 1\) and \( \delta_2\in(0,1) \) such that for every open compact subgroup \( U_f \leq G_{S_f} \) and compact subset \( \scrC \subset Y_S \), there exists \( C_6 = C_6(U_f, \scrC) \) such that for every \( \psi \in C_c^{\infty}(Y_S(\mathfrak{l})) \) with \( \pi^Y_{\mathfrak{l}}\vert_{\mathrm{supp}(\psi)} \) injecting into \( \scrC \), the following holds  for all $g \in G_S$,
		\[
		\normm{
			\int_{Z_S(\mathfrak{l})} \psi(yg) \,\mathrm{d}\rmm^{\bmone}_{Z_S(\mathfrak{l})}(y) -
			\int_{Y_S(\mathfrak{l})} \psi(x) \,\mathrm{d}\rmm^{\bmone}_{Y_S(\mathfrak{l})}(x)
		}
		\leq 
		C_6 \calS^{m_2}(\psi) \left(\prod_{v\in S}\mathrm{Ht}_v(z_0.g) \right)^{- \delta_2} \mathbf{N}(\mathfrak{l})^{\kappa_5}.
		\]
	\end{thm}

	\subsection{Proof of Theorem \ref{thm:effectivemain}}\label{se:pfeffectivemain}
	
	Define \( F_{N}^{\mathfrak{l}}:  Y_S(\mathfrak{l}) \to \mathbb{R}_{\geq 0} \) by
	\[
	F_{N}^{\mathfrak{l}}(o.g) := \sum_{\lambda \in 
		\Gamma(\mathfrak{l}) \cap H_S \bs \Gamma(\mathfrak{l})
	} \bmone_{\BB_S(N)}( z_0. \lambda g)
	\]
	where \( o \) denotes the identity coset in \( Y_S(\mathfrak{l})=\Gamma(\mathfrak{l}) \bs G_S \).
	With this definition,
	$F_{N}^{\mathfrak{l}}(o.\gamma) = \normm{
		z_0.\gamma \Gamma(\mathfrak{l}) \cap \BB_S(N) }
	$. Similarly define 
	$
	F_{N,\ep}^{\mathfrak{l},\pm }(o.g) := \sum_{\lambda \in 
		\Gamma(\mathfrak{l}) \cap H_S \bs \Gamma(\mathfrak{l})
	} \bmone_{B^{\pm}_{N,\ep}}( z_0. \lambda g)
	$.
	By definition we have that for every \( \gamma \in \Gamma \),
	\begin{equation}\label{equation_1}
		\int F_{N,\ep}^{\mathfrak{l},- }(o.\gamma g)  \alpha_{\ep}'(g)\, \mathrm{d}\rmm_{G_S}(g) 
		\leq  F_{N}^{\mathfrak{l} }(o.\gamma )
		\leq 
		\int
		F_{N,\ep}^{\mathfrak{l}, + }(o.\gamma g)  \alpha_{\ep}'(g)\,\mathrm{d}\rmm_{G_S}(g) 
	\end{equation}
	where \( \alpha'_{\varepsilon} :=  
	\left( \int \beta_{\ep}\, \diff\rmm_{H_S}  \right) ^{-1} \alpha_{\ep}\) so that $ \int \alpha_{\ep}'\, \diff\rmm_{G_S}  =1$  with $\alpha_{\ep},\beta_{\ep}$ from Lemma \ref{lemma_E} (take $U_f:= U_f^0$ from Lemma \ref{lemma_A}) for $0<\ep <\ep_3$.
	For convenience the left/right hand side of the above equation will be abbreviated as $\mathrm{LHS}$ or $\mathrm{RHS}$ below.
	Note that 
	\[
	\int \beta_{\ep}\,\diff\rmm_{H_S}  \geq \rmm_{H_S}(H_{\ep^2}) \geq \ep^{\kappa_6}
	\]
	for some $\kappa_6>0$.
	Thus $\calS^m(\alpha'_{\ep}) \leq C_3' \ep^{-\kappa_2'(m)}$ for some $C_3'>0$ and $\kappa_2':\Z^+ \to (0,+\infty)$.
	
	For \( \gamma \in \Gamma \), identifying \( H_{\ep} W_{\ep} \cong \Gamma(\mathfrak{l})\bs\Gamma(\mathfrak{l})\gamma H_{\ep}W_{\ep} \), we define $\alpha_{\ep,\gamma}^{\mathfrak{l}}(o.\gamma h w):= \alpha_{\ep}'(hw)$, a function on \( Y_S(\mathfrak{l})\) with the following properties:
	\begin{enumerate}
		\item \( \int \alpha_{\ep,\gamma}^{\mathfrak{l}} \, \diff\rmm_{Y_S(\mathfrak{l})}  = 1 \),
		\item \( \text{supp}(\alpha_{\ep,\gamma}^{\mathfrak{l}}) \subset  \Gamma(\mathfrak{l})\bs\Gamma H_{\ep}W_{\ep} \),
		\item \( \calS^{m}(\alpha_{\ep,\gamma}^{\mathfrak{l}}) =
		\normm{\rmm_{Y_S(\mathfrak{l})}}^{-1} \calS^m(\alpha_{\ep}') \leq C_3 \ep^{-\kappa_2'(m)} \) with $\kappa_2'(m):=\kappa_2(m)+\kappa_6$.
	\end{enumerate}
	By convention, the implicit measures in the definition of Sobolev norms are normalized to be probability measures whenever it is finite.
	A folding and unfolding argument yields that in Eq. (\ref{equation_1}):
	\begin{equation*}
		\begin{aligned}
			&
			\mathrm{LHS} = \int_{B^-_{N,\varepsilon}} \int_{Z_{S}{(\mathfrak{l})}} \alpha_{\ep,\gamma}^{\mathfrak{l}}(yg) \,\mathrm{d} \rmm_{Z_S(\mathfrak{l})}(y) \,\mathrm{d} \rmm_{X_S}(z_0.g) ,
			\\
			&
			\mathrm{RHS} = \int_{B^+_{N,\varepsilon}} \int_{Z_{S}{(\mathfrak{l})}} \alpha_{\ep,\gamma}^{\mathfrak{l}}(yg) \,\mathrm{d} \rmm_{Z_S(\mathfrak{l})}(y) \,\mathrm{d} \rmm_{X_S}(z_0.g) .
		\end{aligned}
	\end{equation*}

	Let \( \scrC = \Gamma \bs \Gamma H_{\ep_2}W_{\ep_2} \), then \( \psi:= \alpha^{\mathfrak{l}}_{\varepsilon, \gamma} \) satisfies the condition in Theorem \ref{theorem_H} for every \( 0 < \varepsilon < \varepsilon_2 \) and \( \gamma \in \Gamma \). Also note that \( \int \alpha^{\mathfrak{l}}_{\varepsilon, \gamma} \, \diff\rmm^{\bmone}_{Y_S(\mathfrak{l})}  = \normm{
		\rmm_{Y_{S}(\mathfrak{l})}
	} ^{-1} \).
	Applying Theorem \ref{theorem_H} to LHS, we get
	\begin{equation*}
		\begin{aligned}
			&
			\normm{
				\mathrm{LHS} - \frac{  \normm{\rmm_{Z_S(\mathfrak{l})}}
				}{
					\normm{\rmm_{Y_S(\mathfrak{l})}}
				} \rmm_{X_S}(B_{N,\ep}^-)
			}   \\
			\leq & 
			\normm{
				\rmm_{Z_S(\mathfrak{l})}
			} 
			\int_{B^-_{N,\varepsilon}} \left| 
			\normm{  \rmm_{Y_{S}(\mathfrak{l})} } ^{-1} -
			\int_{Z_{S(\mathfrak{l})}}   \alpha^{\mathfrak{l}}_{\varepsilon,\gamma}(yg) \,\mathrm{d} \rmm^{\bmone}_{Z_S(\mathfrak{l})}(y) \right|  
			\,\mathrm{d}\rmm_{X_S}(z_0.g) \\
			\leq &
			\normm{
				\rmm_{Z_S(\mathfrak{l})}
			} C_6 \calS^{m_2}(
			\alpha^{\mathfrak{l}}_{\varepsilon,\gamma}
			)  \mathbf{N}(\mathfrak{l})^{\kappa_5}  \int_{B^-_{N,\varepsilon}}  \left(\prod_{v\in S}\mathrm{Ht}_v(z_0.g) \right)^{-\delta_2} \,\mathrm{d}\rmm_{X_S}(z_0.g).
		\end{aligned}
	\end{equation*}
	Using Lemma \ref{lemma_A}, the above inequalities continue as
	\[
	\leq  \normm{
		\rmm_{Z_S(\mathfrak{l})}
	}
	\cdot C_5 C_3'  \cdot \ep^{-\kappa_2'(m_2)} \cdot \mathbf{N}(\mathfrak{l})^{\kappa_5} C_1 \cdot \rmm_{X_S}(\BB_S(N))^{1-\delta_0(\delta_2)}.
	\]
	We choose constants \( C_7, \, \kappa_7 >1, \, 0<\delta_3<1 \) such that the above quantity is bounded by \( C_7 \mathbf{N}(\mathfrak{l})^{\kappa_7} \ep^{-\kappa_7} \rmm_{X_S}(\BB_S(N))^{1-\delta_3} \).
	Repeating the above process (modifying \( C_6, \, \kappa_7 >0, \, 0<\delta_3<1 \) if necessary), we obtain
	\[
	\normm{
		\mathrm{RHS} - \frac{  \normm{\rmm_{Z_S(\mathfrak{l})}}
		}{
			\normm{\rmm_{Y_S(\mathfrak{l})}}
		} \rmm_{X_S}(B_{N,\ep}^+)
	} \leq  C_7 \mathbf{N}(\mathfrak{l})^{\kappa_7} \ep^{-\kappa_7} \rmm_{X_S}(\BB_S(N))^{1-\delta_3}.
	\]
	Recall from Lemma \ref{lemma_A} that
	$
	\rmm_{X_S}(B^{+}_{N,\varepsilon} \setminus B^{-}_{N,\ep}) \leq C_0 \ep^{\kappa_0} \rmm_{X_S}(\BB_S(N))
	$. 
	Therefore,
	\begin{equation*}
		\begin{aligned}
			&
			\left\vert 
			\normm{z_0 .\gamma \Gamma(\mathfrak{l}) \cap \BB_S(N)} -
			\frac{  \normm{\rmm_{Z_S(\mathfrak{l})} }
			}{
				\normm{ \rmm_{Y_S(\mathfrak{l})} }
			} \rmm_{X_S}(\BB_S(N))
			\right\vert 
			\\
			\leq &
			\normm{
		\mathrm{RHS} - \frac{  \normm{\rmm_{Z_S(\mathfrak{l})}}
		}{
			\normm{\rmm_{Y_S(\mathfrak{l})}}
		} \rmm_{X_S}(B_{N,\ep}^+)
	}
			 + 
			\left| \mathrm{LHS}-  \frac{  \normm{\rmm_{Z_S(\mathfrak{l})} }
			}{
				\normm{ \rmm_{Y_S(\mathfrak{l})} }
			} \rmm_{X_S}(B^-_{N,\ep})
			\right| + 
			\frac{  \normm{\rmm_{Z_S(\mathfrak{l})} }
			}{
				\normm{ \rmm_{Y_S(\mathfrak{l})} }
			}
			\left| \rmm_{X_S}(B^+_{N,\ep})- \rmm_{X_S}(B^-_{N,\ep}) \right| 
			\\
			\leq &
			2C_7 \mathbf{N}(\mathfrak{l})^{\kappa_7} \ep^{-\kappa_7} \rmm_{X_S}(\BB_S(N))^{1-\delta_3} +
			 \frac{  \normm{\rmm_{Z_S(\mathfrak{l})} }
			}{
				\normm{ \rmm_{Y_S(\mathfrak{l})} }
			} C_0 \ep^{\kappa_0} \rmm_{X_S}(\BB_S(N)).
		\end{aligned}
	\end{equation*}
	Fix $\kappa_8 > \kappa_7$ such that $   \frac{  \normm{\rmm_{Z_S(\mathfrak{l})} }
	}{
		\normm{ \rmm_{Y_S(\mathfrak{l})} }
	} \leq \mathbf{N}(\mathfrak{l})^{\kappa_8}$. Choose $\ep:=\ep_N:= \rmm_{X_S}(\BB_S(N))^{-\kappa_9}$ with $\kappa_9:= \delta_3/2\kappa_7$. Then the quantity above is bounded by
	\begin{equation*}
		2C_7 \mathbf{N}(\mathfrak{l})^{\kappa_7}\rmm_{X_S}(\BB_S(N))^{1-\delta_3/2} 
		+ 2C_0  \mathbf{N}(\mathfrak{l})^{\kappa_8} \rmm_{X_S}(\BB_S(N))^{1- \kappa_0\delta_3/ 2 \kappa_7} .
	\end{equation*}
	By setting $C_8:= 2 \max\{C_0, C_7 \}$ and $\delta_4:= \min\{
	\delta_3/2, \kappa_0 \delta_3/2\kappa_7
	\}$, the quantity above is bounded by $C_8 \mathbf{N}(\mathfrak{l})^{\kappa_8}\rmm_{X_S}(\BB_S(N))^{1- \delta_4} $ and the proof is complete, at least for $N $ satisfying $\rmm_{X_S}(\BB_S(N))^{-\kappa_9}<\ep_2$. But it is easy to adjust $C_8,\kappa_8,\delta_4$ such that it also holds for the remaining finitely many $N\in \Z^+$. \qed
	
		\subsection{Sketch of proof of Theorem \ref{thm:GtimesG}}
	The proof being similar to that of Theorem \ref{thm:effectivemain},
	the only modification needed is about Theorem \ref{theorem_B} where one replaces, writing $g=(g_1,g_2)$, $\widetilde{\xi}_S(g)$  by $\widetilde{\xi}^{\max}_S(g_1,g_2):= \max\{\widetilde{\xi}_S(g_1),\widetilde{\xi}_S(g_2) \}$. Again, this does not follow formally from Theorem \ref{theorem_B}, but rather its proof. The details are omitted here.

		\subsection{Consequence on counting integral points with local conditions}
		Let $\CX$ be the Zariski closure of $X_S$ inside $\BA^n_{\CO_S}$. 
		For $\mathfrak{l}$ an ideal of $\CO_{S}$ and $\bxi=(\xi_v)_{v\mid \mathfrak{l}}\in\prod_{v\mid \mathfrak{l}}\CX(\CO_v)$, we consider the following \emph{congruence neighbourhood}  \begin{equation}\label{eq:congneigh}
			\CE_S^{\CX}(\mathfrak{l};\bxi):=\prod_{v\mid \mathfrak{l}}\CE^{\CX}_v(\operatorname{ord}_{v}(\mathfrak{l});\xi_v)\times \prod_{v\nmid \mathfrak{l},v\notin S}\CX(\CO_v),
		\end{equation} where for each $v\mid \mathfrak{l}$ and integer $m$, we let \begin{equation}\label{eq:nuadiccong}
		\CE_v^{\CX}(m;\xi_v):=\left\{x_v\in\CX(\CO_v):\operatorname{Mod}_{v,m}(x_v)=\operatorname{Mod}_{v,m}(\xi_v)\right\},
	\end{equation} where we recall the reduction modulo $\mathfrak{m}_{v}^m$ map \eqref{eq:modvm}, which is a $v$-adic open-closed subset of $\CX(\CO_v)$. 
		It is clear that $\CE_S^{\CX}(\mathfrak{l};\bxi)$ is an open-closed subset of $X(\RA_k^{S\cup\infty_k})$.
		
		Recall that in the work of Borovoi--Rudnick \cite[\S5]{Borovoi-Rudnick}, a locally constant function $\boldsymbol{\delta}:X(\RA_k)\to \{0,\# C(H)\}$ is defined, which captures the adelic orbits of $G(\RA_{k})$ in $X(\RA_{k})$ containing rational points. In the case where $H$ is simply connected, the function $\delta$ is identically one.
		We apply Theorem \ref{thm:effectivemain} to deduce the following effective equidistribution result in terms of congruence neighbourhoods. We recall \eqref{eq:BSN}.
		\begin{theorem}\label{thm:equidistcongruence}
			Let $X=H\bs G$ be a symmetric space as in Theorems \ref{thm:effectivemain} or \ref{thm:GtimesG}. With the notation above, we have, uniformly for $\mathfrak{l}$  and $\bxi$,
			$$\# \left(X(k)\cap \BB_S(N)\times \CE_S^{\CX}(\mathfrak{l};\bxi)\right)=\int_{\BB_S(N)\times X_{\infty_k\setminus S}\times \CE_S^{\CX}(\mathfrak{l};\bxi)}\boldsymbol{\delta}\, \mathrm{d}\rmm_{X}+O_\varepsilon\left(\mathbf{N}(\mathfrak{l})^{\dim G+\kappa+\varepsilon}\cdot \rmm_{X_{S}}(\BB_S(N)) ^{1-\delta}\right)$$ for any $\varepsilon>0$.
		\end{theorem}
\begin{proof}
	The deduction of Theorem \ref{thm:equidistcongruence} is similar to that of \cite[Theorem 5.3]{Borovoi-Rudnick} and its effective version \cite[Corollary 2.5]{Browning-Gorodnik}. By work of Borel--Harish-Chandra \cite{Borel-HC}, there are only finitely orbits of $\Gamma\cap G(k)$ in $\CX(\CO_S)$. As $X_S$ itself is a $G_S$-orbit, the number of adelic orbits of $G(\RA_{k})$ in $X(\RA_{k})$ intersecting $X_S\cap\CX(\CO_S)$ is thus also finite. We fix an ideal $\mathfrak{l}_{0}$ such that $\mathfrak{l}\mid\mathfrak{l}_0$ and for every such adelic orbit $\mathbf{O}_{\RA}=X_S\times\mathbf{O}_{\RA}^S$ (so that $\boldsymbol{\delta}|_{\mathbf{O}_{\RA}}\equiv \# C(H)$),  the set $\CE_S^{\CX}(\mathfrak{l};\bxi)\cap \mathbf{O}_{\RA}^S$ (which is open by \cite[Lemma 1.6.4]{Borovoi-Rudnick}) is invariant under the action of the group $\Gamma(\mathfrak{l}_0)$ (the principal congruence subgroup modulo $\mathfrak{l}_0$ \eqref{eq:gammal}). Let $\scrO$ be such an orbit of $\Gamma(\mathfrak{l}_0) $ in $\CE_S^{\CX}(\mathfrak{l};\bxi)\cap \mathbf{O}_{\RA}^S\cap X(k)$. The number of such $\scrO$ is $\leq \#(\Gamma(\mathfrak{l}_0)\bs\Gamma) \ll_\varepsilon \mathbf{N}(\mathfrak{l})^{\dim G+\varepsilon}$ by the Lang--Weil estimate (see e.g. \cite[Theorem 3.5]{Cao-Huang1}). Fixing $z_0\in\scrO$, we infer from Theorem \ref{thm:effectivemain} that 
	$$\# \left(\scrO\cap \BB_S(N)\right)=\frac{\rmm_{H_S}(Z_S(\mathfrak{l}))}{\rmm_{G_S}(Y_S(\mathfrak{l}))}\rmm_{X_S}(\BB_S(N))+O(\mathbf{N}(\mathfrak{l})^{\kappa} \cdot \rmm_{X_{S}}(\BB_S(N)) ^{1-\delta}),$$ where $\kappa>0,0<\delta<1$ are independent of $\scrO,\mathfrak{l},N$.
	Summing over all such orbits $\scrO$ and appealing to the Siegel--Weil formulae (Theorem \ref{thm:BRSiegelWeil}), we obtain 
	\begin{multline*}
		\# \left(X(k)\cap \BB_S(N)\times (\CE_S^{\CX}(\mathfrak{l};\bxi)\cap \mathbf{O}_{\RA}^S)\right)\\=\#C(H)\rmm_X\left(\BB_S(N)\times X_{\infty_k\setminus S}\times (\CE_S^{\CX}(\mathfrak{l};\bxi)\cap \mathbf{O}_{\RA}^S)\right)+O_\varepsilon\left(\mathbf{N}(\mathfrak{l})^{\dim G+\kappa+\varepsilon}\cdot \rmm_{X_{S}}(\BB_S(N)) ^{1-\delta}\right).
	\end{multline*}
	It remains to sum over all such adelic orbits $\mathbf{O}_{\RA}$ to obtain the desired formula.
\end{proof}

\begin{corollary}
	The symmetric space $X=H\bs G$ as in Theorems \ref{thm:effectivemain} or \ref{thm:GtimesG} is Hardy--Littlewood off any single place over which $G$ is not compact.
\end{corollary}
This should already be implicitly known though we could not find it in the literature.

\subsection{Example -- Affine quadrics}\label{se:affquavol}
    Let $q$ be a non-degenerate quadratic form in $n$ variables with coefficients in $k$ and $m\in k^{\times}$. Let $$X:=(q=m)\subset \BA^n_k$$ be an affine quadric. Let $v$ be a place of $k$. Assume that $q$ is $k_v$-isotropic and $X(k_v)\neq \varnothing$. Define $$\BB_v(N):= \left\{\bx\in k_v^n \midd  q(\bx)=m,\; \mathrm{Ht}_v(\bx)\leq N \right\}.$$ Then there exist $0<c<C<\infty$ such that \begin{equation}\label{eq:affquavol}
        cN^{n-2}< \rmm_X(\BB_v(N))< CN^{n-2}\quad\text{as}~ N\to\infty.
    \end{equation}

The growth of volume balls is intimately related via Tauberian-type theorems (cf. e.g. \cite[Appendix A]{Chamber-Loir_Tschinkel_2010_Igusa_integral}) to the height zeta function
\[
 Z(s):= \int_{\bx\in X(k_v)} \mathrm{Ht}^{-s}_{v}(\bx) \, \mathrm{d}\normm{\omega}_{v}(\bx)
\]
which admits a real pole at $s=n-2$. Note that $n-1=n-2+1$ is the order of the pole of the natural volume form on the affine quadric $X$ along the boundary divisor that compactifies it into the projective space $\mathbb{P}^n$. In order to keep our paper as self-contained as possible, we choose to give details for proof of \eqref{eq:affquavol} in the case where $v$ is non-archimedean, as it will be used in \S\ref{se:geomsieve}. The archimedean case works \emph{verbatim}.

We recall that the (normalized) Tamagawa measure $\rmm_X$ is induced from the a gauge (volume) form $\omega$ on $X$ satisfying
\[
   \omega \wedge \diff q = \diff x_1 \wedge \cdots \wedge \diff x_n.
\] We may diagonalize the quadratic form $q$ (as this only affects the height balls up to a bounded constant) so that 
\[
 q(x_1,\cdots,x_n) = a_1x_1^2 +\cdots +a_n x_n^2
\]
with $a_1,\cdots,a_n \in k^{\times}$.
So when $x_i \neq 0,1\leq i\leq n$, we have
\begin{equation}\label{eq:omega}
      \omega = \pm\frac{1}{2a_ix_i} \diff x_1\wedge\cdots\wedge \diff x_{i-1}\wedge \diff x_{i+1} \wedge\cdots\wedge \diff x_n.
\end{equation}

We now embed $\mathbb{A}^n \embed \mathbb{P}^n$ via
$(x_1,\cdots,x_n)\mapsto [1: x_1:\cdots:x_n]$.  Then the height function \eqref{eq:height} extends to $\BP^n$ so that for $\by=[y_0:\cdots:y_n]$ with $y_0\neq 0$ (note that $x_i=\frac{y_i}{y_0},1\leq i\leq n$), $$\mathrm{Ht}_{v}(\by) =
     \frac{\max_{i=1,\cdots,n}\normm{y_i}_{v}}
    {\normm{y_0}_{v}}.$$
Let $$\overline{X}=\left(\sum_{i=1}^n a_i y_i^2 = m y_0^2\right)\subset\BP^n$$ be the Zariski closure of $X$ in $\mathbb{P}^n$.
The boundary divisor $\partial X := \overline{X} \setminus X$ is the intersection of the hyperplane $y_0=0$ with $\overline{X}$, which is thus the projective quadric 
     $$\sum_{i=1}^n a_i y_i^2 = 0.$$
 
We wish to study in more detail the behavior of $Z(s)$ near the boundary divisor $\partial X$. Let $P\in \partial X(k_v)$. We may assume that $y_n(P)\neq 0$, and introduce the coordinate system $z_i=\frac{y_i}{y_n},i=0,1,\cdots,n-1$. By implicit function theorem, if we choose $\bz=(z_0,\cdots,z_{n-2})$ as a local coordinate system for $\overline{X}(k_v)$ near $P\in \partial X(k_v)$, there exists a bounded $v$-adic neighborhood $\CN_P\subset k_v^{n-1}$ of $\bz(P)=(0,z_1(P),\cdots,z_{n-2}(P))$ and an analytic function $u_P: \CN_P \to k_{v}$ such that for all $\bz\in \CN_P$, $[\bz:u_P(\bz):1]\in\overline{X}(k_v)$, i.e., $z_{n-1}$ is solved by $u_P$. Then the volume form $\omega$ \eqref{eq:omega} restricted to $\CN_P\cap (y_0\neq 0)\cap (y_n\neq 0)$ is 
$$\omega  =
    \pm \frac{z_{0}}{2a_n} 
    \left( \bigwedge_{j=1}^{n-2} \diff\left(\frac{z_{j}}{z_{0}}\right) \right)\wedge \diff\left(\frac{u_P(\bz)}{z_{0}}\right)
    = \pm  \frac{u_P(\bz)}{2a_n z_{0}^{n-1}} 
    \left( \bigwedge_{j=0}^{n-2}\diff{z_{j}} \right).$$ 
    
Let us consider the height zeta function restricted to  $\CN_P$: $$Z_{P}(s)
        :=\int_{\bx\in X(k_{v}) \cap \CN_P} \mathrm{Ht}_{v} ^{-s}(\bx)\,  \mathrm{d}\normm{\omega}_{v}(\bx).$$ Then
\begin{equation*}
    \begin{aligned}
     Z_{P}(s) &= \int_{\bz\in \CN_P} \left(\frac{\max(|z_i|_{v},1\leq i\leq n-1,1)}{|z_0|_{v}}\right)^{-s} \left|\frac{u_P(\bz)}{2a_n z_0^{n-1}}\right|_{v}\mathrm{d}\bz \\ &=\int_{\bz\in \CN_P}\phi_s(\bz)|z_0|_{v}^{s-n+1}\mathrm{d}\bz,
    \end{aligned}
\end{equation*}
where for simplicity we write 
\[
   \phi_s(\bz):= \frac{
           |u_P(\bz)|_{v}
        }
        {|2a_n|_{v}
            \max(|z_i|_{v},1\leq i\leq n-2,|u_P(\bz)|_{v},1)
        }.
\]
By shrinking $\CN_P$, we may assume that  $\phi_s: \CN_P \to \R_{>0}$ is bounded away from zero and infinity uniformly for $s$ in bounded strips. As $v$ is a non-archimedean place, we may further assume that $\CN_P$ takes the form (let $\pi_{v}$ be the uniformizer of the local ring $\mathcal{O}_v$)
\[
\calN_P = \pi_{v}^{r_P} \mathcal{O}_{v} \times \calN'_P
\]
for some $r_P \in \Z^+$ and bounded $\calN'_P\subset k_v^{n-2}$.
Now 
\begin{align*}
      Z_P(s) &= \int_{(z_0,\bz')\in\pi_{v}^{r_P} \CO_{v} \times \calN'_P} \normm{z_0}_{v}^{s-n+1} \phi_s(z_0,\bz') \, \diff z_0 \diff \bz'\\ &= 
         \int_{\bz'\in\calN_P'} \phi_s(0,\bz') \left(
         \int_{z_0\in\pi_{v}^{r_P} \CO_{v}} \normm{z_0}_{v}^{s-n+1}  \,\diff z_0
         \right) \diff \bz'+
          \int_{\bz\in\CN_P} \normm{z_0}_{v}^{s-n+2} 
          \frac{
              \phi_s(z_0,\bz') - \phi_s(0,\bz')
          }{\normm{z_0}_{v}
          }\, \diff \bz.
\end{align*}
 Note that
\[
   \int_{\pi_{v}^{r_P} \CO_{v}}   \normm{z_0}_{v}^{s-n+1}   \, \diff z_0
    = 
   \sum_{r\geq r_P} \normm{\pi_{v}}_{v}^{-r}  \cdot \normm{\pi_{v}}_{v}^{-r(s-n+1)}
   = \frac{
     \normm{\pi_{v}}_{v}^{-r_P(s-n+2)}
   }{1 - \normm{\pi_{v}}_{v}^{-s+n-2}}
\]
is meromorphic in $s\in \C$ with poles of order $1$ at $s= n-2+ \frac{2\pi i }{\log\normm{\pi_{v}}_{v}} \cdot \Z$.
On the other hand, the function  $\frac{\phi_s(z_0,\bz') - \phi_s(0,\bz')}{\normm{z_0}_{v}}$ 
is bounded. So the second term is absolutely convergent and analytic for all $\mathrm{Re}(s)>n-3$ and is bounded in vertical strips.
By applying a suitable Tauberian theorem (see the discussion above \cite[Corollary A.4]{Chamber-Loir_Tschinkel_2010_Igusa_integral} for details), we obtain that  $$\rmm_X(\BB_v(N) \cap \calN_P)\asymp N^{n-2}.$$  Covering $\overline{X}(k_v)$ by finitely many such $\calN_P$'s together with some bounded (compact) set in $X(k_v)$, we find that $0<c<\rmm_X(\BB_v(N))/N^{n-2} <C<\infty$ for suitable constants $c,C$.

\subsection{Volume asymptotics after Chambert-Loir and Tschinkel}\label{se:volasy}
The goal of this section is to sketch a general framework due to Chambert-Loir--Tschinkel  \cite{Chamber-Loir_Tschinkel_2010_Igusa_integral} concerning growth of volume balls for varieties over local fields  and giving a geometric interpretation of the exponents $a,b$.

Let $\mathscr{X}$ be a smooth projective variety over $k$, and  $\omega_{\mathscr{X}}$ be the canonical line bundle. In the setting of \cite[Section 4.2]{Chamber-Loir_Tschinkel_2010_Igusa_integral}, let $(\mathscr{D}_{\alpha})_{\alpha \in \mathscr{A}}$ be a finite collection of irreducible reduced divisors, and let $\mathscr{U}$ be the complement of $\cup_{\alpha\in \mathscr{A}} \mathscr{D}_{\alpha}$ in $\mathscr{X}$. Let $(d_{\alpha})_{\alpha \in \mathscr{A}}$ be a collection of positive integers that defines the divisor $\mathscr{D}:=\sum_{\alpha \in \scrA} d_{\alpha} \mathscr{D}_{\alpha}$. Let $\mathscr{L}$ be an effective  divisor  on  $\mathscr{X}$ such that the support of $\mathscr{D}$ is contained in that of $\mathscr{L}$. 

Let us fix $S$ a finite set of places of $k$. We fix a choice of smooth metric $(\norm{\cdot}_{\mathscr{O}_{\mathscr{X}}(\mathscr{L}),v})_{v\in S}$ for the line bundle $\mathscr{O}_{\mathscr{X}}(\mathscr{L})$, and define for $v\in S$ the height function $\mathrm{Ht}_v:\scrX(k_v)\to \BR_{>0}$ as \begin{equation}\label{equation_height} \mathrm{Ht}_v(\bx):= \norm{\bmone_{\mathscr{L}}(\bx)}_{\mathscr{O}_{\mathscr{X}}(\mathscr{L}),v}^{-1}.\end{equation} Equipping $\omega_{\mathscr{X}}(\mathscr{D})$ with a smooth metric $(\norm{\cdot }_{\omega_{\mathscr{X}}(\mathscr{D}),b})_{v\in S}$,
let $\tau_{(\mathscr{X},\mathscr{D}),v}$ be the measure on $\scrX(k_v)$ locally induced by the volume form
\[
\frac{\theta}{\norm{ 
(\theta\otimes \bmone_{\mathscr{D}})
}_{\omega_{\mathscr{X}}(\mathscr{D}),v}
},
\]
where $\theta$ is a nonvanishing local section of $\omega_{\mathscr{X}}$.
With all these data, we associate the height function 

   $$ Z(s):= \int_{\bx=(x_v)\in\mathscr{X}_S}  \left(\prod_{v\in S}\mathrm{Ht}_v(x_v)\right)^{-s}  \,\diff \left(\prod_{v\in S}\tau_{(\mathscr{X},\mathscr{D}),v}(x_v)\right).$$

Assuming that $\mathscr{D}$ is a strictly normal crossings divisor over the the algebraic closure, the volume growth can be read from the corresponding analytic Clemens complex. Concretely, define two constants $a,b$ as follows:
\begin{itemize}
    \item For each $v\in S$, $a_v:= \max_{\alpha \in \mathscr{A}} \left\{
      \frac{d_{\alpha}-1}{\lambda_{\alpha}}
      \midd  \mathscr{D}_{\alpha}(k_v)\neq \varnothing
    \right\}$;
    \item For each $v\in S$, let $\mathscr{A}_0\subset \mathscr{A}$ be the subset where the above $\max$ is attained and $b_v$ be the maximal number such that there exist $\beta_1,...,\beta_{b_v}\in \mathscr{A}_0$ with $\cap_{i=1}^{b_v} \scrD_{\beta_i}(k_v)\neq \varnothing$. 
    \item $a:= \max\{a_v ,\; v\in S\}$ and $b:=\sum_{v\in S,a_v=a} b_v$.
\end{itemize}
For $N>1$, consider the volume ball $$\BB_S(N):=\left\{\bx=(x_v)\in \scrX_S\midd \prod_{v\in S}\mathrm{Ht}_v(x_v)\leq N \right\}.$$
Then the results of \cite[\S4 \& Appendix]{Chamber-Loir_Tschinkel_2010_Igusa_integral} show that (in the case of $a >0$) $Z(s)$ has a meromorphic continuation to $\Re(s)>a-\delta$ for any $\delta>0$ with a pole of order $b$ at $s=a$.
This implies that there exist some $0<c<C<\infty$ such that for $N$ sufficiently large
\begin{equation}\label{eq:Svol}
        cN^{a}(\log N)^{b-1}< \left(\prod_{v\in S}\tau_{(\mathscr{X},\mathscr{D}),v}\right)(\BB_S(N))< CN^{a}(\log N)^{b-1}.
\end{equation}

\begin{example}[Affine quadrics]
Recalling the notation used in \S\ref{se:affquavol},  $\mathscr{U}$ is the affine quadric $X$, $\mathscr{X}$ is the compactification $\overline{X}$ and $\mathscr{A}$ is a singleton $\{\alpha_0\}$ with $\mathscr{D}_{\alpha_0}= \partial X$. 
Let $\mathscr{L}=\partial X = \mathscr{D}_{\alpha_0}$, and  $\mathscr{O}_{\mathscr{X}}(\mathscr{L})$ be equipped with the metric that coincides with \eqref{eq:height}.  
   Let $d_{\alpha_0}:=n-1$ and $\scrD := (n-1)\scrD_{\alpha_0}$. 
The measure $\tau_{(\mathscr{X},\mathscr{D})}$ is induced by the gauge form $\omega$ on $X$. We see that $$a=(n-1)-1=n-2,\quad b=1,$$ and we thus recover \eqref{eq:affquavol}. 
Indeed, the calculation in the previous section can be regarded as a special case of \cite[\S5.4]{Chamber-Loir_Tschinkel_2010_Igusa_integral}.\end{example} 

\begin{example}[Semisimple groups]
Let $G$ be a semisimple simply connected linear algebraic group over $k$. We give a brief sketch for the volume asymptotics on $G$  in the case where $G$ is $k$-split. We refer to \cite[\S 5.2]{Chamber-Loir_Tschinkel_2010_Igusa_integral} for more details about general cases. 

Let us fix $T\subset B \subset G$ where $B$ is a Borel $k$-subgroup and $T$ is a maximal $k$-split torus contained in $B$. We let $\Delta=\{\alpha_1,...,\alpha_l\}$ (so $l$ is the rank of $G$) be the corresponding simple roots.
Let $\lambda^+$ be a regular dominant weight and let $\psi_{\lambda^+}: G \to \mathrm{GL}(V_{\lambda^+})$ be the associated irreducible representation. Let $G^{\mathrm{ad}}$ be the adjoint quotient of $G$.
Then $\psi_{\lambda^+}$ induces a projective representation $\overline{\psi_{\lambda^+}}: G^{\mathrm{ad}} \to \mathrm{PGL}(V_{\lambda^+}) \subset \BP(\mathrm{End}(V_{\lambda^+}))$.
We fix a basis of $V_{\lambda^+}$ consisting of $T$-eigenvectors so that 
we identify $\BP(\mathrm{End}(V_{\lambda^+}))$ with $\BP^{n^2-1}$, where $n=\dim V_{\lambda^+}$. 
Let $L_0$ be the divisor $\{\det=0\}$ in $\BP(\mathrm{End}(V_{\lambda^+}))$. We equip $\mathscr{O}_{\BP^{n^2-1}}(L_0)$, which is isomorphic to $\mathscr{O}(n)$, with a smooth metric such that 
\[
    \norm{\bmone_{L_0} (
    \overline{g}
    ) }_{\mathscr{O}(n),v}^{-1} = 
    \norm{ g }_v,
    \quad
    \forall \, g\in \SL(V_{\lambda^+})(k_v)
\]
In particular,
\[
    \norm{\bmone_{L_0} (
    \overline{\psi_{\lambda^+}}(\overline{g})
    ) }_{\mathscr{O}(n),v}^{-1} = 
    \norm{\psi_{\lambda^+}(g )}_v,
    \quad
    \forall \, g\in G(k_v)
\]
where $\overline{g}$ is the image of $g$ in $G^{\mathrm{ad}}$.

Let  $\mathscr{X}$ be the Zariski closure of $G^{\mathrm{ad}}$ in $\BP(\mathrm{End}(V_{\lambda^+}))$. The work of De Concini--Procesi \cite{DeConcini_Procesi_symmetric} shows that $\mathscr{X}$ is smooth with strictly normal crossing boundary $\sum_{i=1}^l \mathscr{D}_i$ labelled by the simple roots. Moreover,
let $\beta$ be the sum of all positive roots, and define $(m_i)_{1\leq i\leq l}$ by
\[ 
 \beta = \sum_{i=1}^l m_i \alpha_i.
\]
Also define $(d_i)_{1\leq i\leq l}$ by
\[
   \lambda^+ = \sum_{i=1}^l  d_i \alpha_i.
\] 
Note that $\lambda^+$ being regular implies that $d_i>0$ for all $i$.
Then, we can choose local coordinates $(x_1,\cdots,x_{\dim G})$ such that $\mathscr{D}_i=(x_i=0),1\leq i\leq l$, $\mathscr{L}:=L_0|_{\mathscr{X}}=\sum_{i=1}^l d_i \mathscr{D}_i$, and an gauge form on $G^{\mathrm{ad}}$ can be chosen as $$\omega = \frac{\diff x_1\wedge\cdots\wedge\diff x_{\dim G}}{\prod_{i=1}^l x_i^{m_i+1}}.$$ Setting $\mathscr{D}:=\sum_{i=1}^l (m_i+1)\mathscr{D}_i$ and pulling back the metric above to $\mathscr{X}$,
then we see that
\begin{equation*}
\begin{aligned}
     a= \max\left\{
    \frac{m_i}{d_i} \midd i=1,\cdots,l
    \right\},\quad
    b= \#\left\{ i=1,\cdots,l \midd a= \frac{m_i}{d_i}
    \right\}.
\end{aligned}
\end{equation*}
\end{example}

		\section{The affine linear sieve for $S$-integral points}\label{se:affinelinearsieve}
		
		The goal of this section is to establish the following version of the \emph{affine linear sieve} over for $S$-integral points over any number field, generalizing work of Nevo--Sarnak \cite[\S4]{N-S} (cf. also work of Gorodnik--Nevo \cite[\S1.3]{G-N}) and previous work of the first two authors \cite[\S3]{Cao-Huang1}.
		It serves as a crucial ingredient for executing the fibration method in proving \APSA\ off any single place.
		
		\begin{theorem}\label{thm:almostprime}
			Consider the following data:
			\begin{itemize}
				\item 	Let $G$ be a semisimple simply connected $k$-simple linear algebraic group over $k$ and let $\CG$ be the Zariski closure of $G$ in $\GL_{n,\CO}$;
				\item  Let $S$ be a finite set of places containing the archimedean ones, and $v_0\in S$ be such that $G(k_{v_0})$ is not compact;
				\item  Let $\mathfrak{D}\subset G(k_{S\setminus \{v_0\}})$ be a compact region with non-empty interior and piecewise smooth boundary;
				\item 
				Let $S'\supset S$ be a finite set of places and let $$\Phi:=\prod_{v\in S'\setminus S}\Phi_v\times \prod_{v\notin S'}\CG(\CO_v)$$ be a subgroup of $G(\RA_k^{S})$, where $\Phi_v\sbt G(k_{v})$ is an open compact subgroup for each $v\in S'\setminus S$;
				\item 	Let $f\in \CO_{S'}[\CG]$ be such that $f(\CG(\CO_v))\cap \CO_v^{\times}\neq\varnothing$ for any $v\in \Omega_{k}\setminus S'$. 
			\end{itemize}
			Then there exists an integer $r_0\geq 1$ depending only on $\CG,\mathfrak{D},\Phi,f,S,S^\prime$ and satisfying the following:
			
			For any $M>0$, there exists $g_0\in G(k)\cap \mathfrak{D}\cap \Phi$ 
			such that $f(g_0)$ is divisible by at most $r_0$ places not in $S'$, and the cardinalities of their residue fields are larger than $M$.
		\end{theorem}
	
	\subsection{A combinatorial sieve over general number fields}
		Let $k$ be a number field, and $S_0$ be a finite set of places of $k$ containing all the archimedean ones. 
		Let $\CI_{S_0}$ be the set of square-free ideals of $\CO_{S_0}$.
		For $\CP$ any set of prime ideals of $\CO_{S_0}$ and $z>0$, we use the notation $$\CP(z):=\prod_{\mathfrak{p}\in\CP:\mathbf{N}(\mathfrak{p})<z}\mathfrak{p}.$$
		Let $\CA=(a_{\mathfrak{b}})_{\mathfrak{b}\in I}$ be a sequence of non-negative real numbers indexed by a finite set $I$ of ideals of $\CO_{S_0}$. The \emph{sifting function} (with respect to $\CP$ and $z$) is
		\begin{equation}\label{eq:sifting}
			\CS(\CA,\CP,z):=\sum_{\substack{\mathfrak{b}\in I\\\gcd(\mathfrak{b},\CP(z))=1}}a_\mathfrak{b}.
		\end{equation} The cardinality of $\CA$ is defined as $$\#\CA:=\sum_{\mathfrak{b}\in I}a_{\mathfrak{b}}.$$ For every ideal $\mathfrak{d}\in\CI_{S_0}$, we define the subsequence $$\CA_{\mathfrak{d}}:=(a_\mathfrak{b})_{\mathfrak{b}:\mathfrak{d}\mid \mathfrak{b}}$$ of $\CA$. 
		
		We now state the following version of Brun's combinatorial sieve over the number field $k$.  See e.g. \cite[Theorem 6.9, Corollary 6.10]{Friedlander-Iwaniec} when $k=\BQ$ and \cite[Lemma 4]{Hinz} for a version over general number fields.
		\begin{theorem}[``Fundamental Lemma'' of the combinatorial sieve]\label{thm:sieve}
			With the notation above, let $\omega:\CI_{S_0}\to\BR_{\geq 0}$ be a multiplicative function satisfying: \begin{itemize}
				\item For every prime ideal $\mathfrak{p}$ of $\CO_{S_0}$, we have $0\leq\frac{\omega(\mathfrak{p})}{\mathbf{N}(\mathfrak{p})}<1$;
				\item There exist constants $\kappa_1>0, \kappa_2>1$ such that
				for any real numbers $2\leq w_1\leq w_2$, we have
				\begin{equation}\label{eq:sievedim}
					\prod_{\substack{w_1\leq \mathbf{N}(\mathfrak{p})\leq w_2}}\left(1-\frac{\omega(\mathfrak{p})}{\mathbf{N}(\mathfrak{p})}\right)^{-1}\leq \kappa_2\left(\frac{\log w_2}{\log w_1}\right)^{\kappa_1}.
				\end{equation}
			\end{itemize}
		Then there exist constants $\lambda,\tau>0$ depending only on $\kappa_1,\kappa_2$ such that, for any finite sequence of non-negative real numbers $\CA=(a_{\mathfrak{b}})_{\mathfrak{b}\in I}$, and $X>0,y\geq 2,z\in[2,y^\lambda]$, we have
		\begin{equation*}
			\CS(\CA,\CP,z)\geq \tau X\prod_{\mathfrak{p}\mid\CP(z)}\left(1-\frac{\omega(\mathfrak{p})}{\mathbf{N}(\mathfrak{p})}\right)+O\left(\sum_{\substack{\mathfrak{d}\in\CI_{S_0}\\\mathbf{N}(\mathfrak{d})\leq y,\mathfrak{d}\mid \CP(z)}}\left|\#\CA_{\mathfrak{d}}-\frac{\omega(\mathfrak{d})}{\mathbf{N}(\mathfrak{d})}X\right|\right),
		\end{equation*}
		where the implied constant depends only on $\kappa_1,\kappa_2$, and is uniform with respect to the sequence $\CA$, the set $\CP$ and the real numbers $X,y,z$.
		\end{theorem}

\subsection{Proof of Theorem \ref{thm:almostprime}}
	We keep using the notation in \S\ref{se:effective}. For any ideal $\mathfrak{l}$ of $\CO_{S}$, recall from \eqref{eq:gammal} that $\Gamma(\mathfrak{l})$ denotes a principal congruence subgroup.
	 We fix an ideal $\mathfrak{l}_0$ such that $\Gamma(\mathfrak{l}_0)\subset \Phi$. Let $$S_0:=S'\cup \{v:v\mid\mathfrak{l}_0\}.$$ We may also assume that $\CG$ is smooth over $\CO_{S_0}$ upon enlarging $S_0$. We write $\rmm_G^S$ (resp. $\rmm_G^{S_0}$) for the induced Tamagawa measure on $G(\RA_{k}^S)$ (resp. $G(\RA_{k}^{S_0})$).
	 
	 Let $N>1$ be sufficiently large. 
	 We let $I(N)$ be the collection of ideals of $\CO_{S_0}$ generated by values of $f$ on the (finite) set $G(k)\cap \mathfrak{D}\times B_{\{v_0\}}(N)\times\Gamma(\mathfrak{l}_0)$ and we form the finite sequence $\CA(N)=\{a_{\mathfrak{b}}\}_{\mathfrak{b}\in I(N)}$ where \begin{equation}\label{eq:seqab}
	 	a_{\mathfrak{b}}:=\#\{g\in G(k)\cap \mathfrak{D}\times B_{\{v_0\}}(N)\times \Gamma(\mathfrak{l}_0) :f(g)\CO_{S_0}=\mathfrak{b}\}.
	 \end{equation}
	 
	 For any ideal $\mathfrak{d}\in\CI_{S_0}$, we write  
	 $$\CE_{S_0}(\mathfrak{d}):=\prod_{v\notin S_0}\{g_v\in\CG(\CO_v):\operatorname{ord}_{v}(\mathfrak{d})\leq \ord_{v}(f(g_v))\}\subset G(\RA_k^{S_0}),$$ $$\CE_{S}(\mathfrak{d}):=\prod_{v\in S_0\setminus S}\CE_v^{\CG}(\operatorname{ord}_{v}(\mathfrak{l}_0);\operatorname{id}_{v})\times \CE_{S_0}(\mathfrak{d})\subset G(\RA_k^S),$$ where $\operatorname{id}_{v}$ denotes the identity element of $\CG(\CO_{v})$ and $\CE_v^{\CG}(\operatorname{ord}_{v}(\mathfrak{l}_0);\operatorname{id}_{v})$ is a $v$-adic congruence neighbourhood of $\operatorname{id}_{v}$ \eqref{eq:nuadiccong}. 
	  We decompose $\CE_{S}(\mathfrak{d})$ into a finite disjoint union of congruence neighbourhoods of level $\mathfrak{l}_0\mathfrak{d}$: $$\CE_S(\mathfrak{d})= \bigsqcup_{j\in J} \CE_S^{\CG}(\mathfrak{l}_0\mathfrak{d};\bxi_j),$$ the number of which is $\# J\leq \mathbf{N}(\mathfrak{l}_0\mathfrak{d})^{\dim G+\varepsilon}$ by the Lang--Weil estimate.
	 
	 We now apply Theorem \ref{thm:equidistcongruence} (note that $\boldsymbol{\delta}\equiv 1$ in our case) to each of these congruence neighbourhoods and get 
	 \begin{equation}\label{eq:CAN}
	 		 \begin{split}
	 		\#\CA(N)_{\mathfrak{d}}&=\#\{g\in G(k)\cap \mathfrak{D}\times B_{\{v_0\}}(N)\times \Gamma(\mathfrak{l}_0) :\mathfrak{d}\mid f(g)\}\\ &=\sum_{j\in J}\left(\rmm_{G}(\mathfrak{D}\times B_{\{v_0\}}(N)\times \CE_S^{\CG}(\mathfrak{l}_0\mathfrak{d};\bxi_j))+O_\varepsilon\left(\mathbf{N}(\mathfrak{d})^{\dim G+\kappa+\varepsilon}\rmm_{G_{v_0}}(B_{\{v_0\}}(N))^{1-\delta}\right)\right)\\ &=\rmm_{G_S}(\mathfrak{D}\times B_{\{v_0\}}(N))\rmm_{G}^{S}(\CE_S(\mathfrak{d}))+O_\varepsilon\left(\mathbf{N}(\mathfrak{d})^{2\dim G+\kappa+\varepsilon}\rmm_{G_{v_0}}(B_{\{v_0\}}(N))^{1-\delta}\right).
	 	\end{split}
	 \end{equation}
We write 
$$\iota_1:=\prod_{v\in S_0\setminus S}\rmm_{G_v}(\CE_v^{\CG}(\operatorname{ord}_{v}(\mathfrak{l}_0);\operatorname{id}_{v})),\quad \iota_2:=\rmm_{G}^{S_0}\left(\prod_{v\notin S_0}\CG(\CO_{v})\right).$$ The closed subset $V:=(f=0)\cap G$ has codimension one. Let $\CV$ be its Zariski closure in $\CG$. By the usual computation of Tamagawa measures (cf. e.g. \cite[Theorem 2.14]{Salberger}), we have
\begin{equation}\label{eq:mGS}
	\begin{split}
		\rmm_{G}^{S}(\CE_S(\mathfrak{d}))&=\iota_1 \rmm_{G}^{S_0}(\CE_{S_0}(\mathfrak{d}))\\ &=\iota_1 \prod_{\substack{v\notin S_0,v\nmid \mathfrak{d}}} \rmm_{G_{v}}(\CG(\CO_{v}))\times\prod_{\substack{v\notin S_0,v\mid\mathfrak{d}}}\rmm_{G_{v}}(\{g_{v}\in\CG(\CO_{v}):g_{v}\bmod \mathfrak{p}_{v} \in \CV\})\\ &=\iota_1\iota_2\prod_{\substack{v\notin S_0,v\mid\mathfrak{d}}}\frac{\rmm_{G_{v}}(\{g_{v}\in\CG(\CO_{v}):g_{v}\bmod \mathfrak{p}_{v} \in \CV\})}{\rmm_{G_{v}}(\CG(\CO_{v}))}\\ &=\iota_1\iota_2\prod_{\substack{v\notin S_0,v\mid\mathfrak{d}}}\frac{\#\CV(\BF_{\mathfrak{p}_{v}})}{\#\CG(\BF_{\mathfrak{p}_{v}})}.
	\end{split} 
\end{equation}
Let the multiplicative function $w:\CI_{S_0}\to\BR_{\geq 0}$ be defined by
$$\omega(\mathfrak{p}):=\mathbf{N}(\mathfrak{p})\frac{\#\CV(\BF_{\mathfrak{p}})}{\#\CG(\BF_{\mathfrak{p}})}$$ and let $$\boldsymbol{X}(N):=\iota_1\iota_2 \rmm_{G_S}(\mathfrak{D}\times B_{\{v_0\}}(N)).$$ Then our computations \eqref{eq:CAN} \eqref{eq:mGS} obtained so far show that for all $\mathfrak{d}\in\CI_{S_0}$,
\begin{equation}\label{eq:leveldist}
	\#\CA(N)_{\mathfrak{d}}-\frac{\omega(\mathfrak{d})}{\mathbf{N}(\mathfrak{d})}\boldsymbol{X}(N)=O_\varepsilon\left(\mathbf{N}(\mathfrak{d})^{2\dim G+\kappa+\varepsilon} \boldsymbol{X}(N)^{1-\delta}\right).
\end{equation}

We next proceed to verify the condition \eqref{eq:sievedim} on sieve dimension.  The Lang--Weil estimate (see e.g. \cite[Theorem 3.5]{Cao-Huang1}) applied to $\CG,\CV$ implies that there exist absolute constants $A_{1},A_{2}>0$ such that for any prime $\mathfrak{p}\notin S$,
 $$\#\CG(\BF_{\mathfrak{p}})\geq A_{1}\mathbf{N}(\mathfrak{p})^{\dim G},\quad \#\CV(\BF_{\mathfrak{p}})\leq A_{2}\mathbf{N}(\mathfrak{p})^{\dim G-1},$$ whence $$\frac{\omega(\mathfrak{p})}{\mathbf{N}(\mathfrak{p})}=\frac{\#\CV(\BF_{\mathfrak{p}})}{\#\CG(\BF_{\mathfrak{p}})}\leq \frac{\kappa_1}{\mathbf{N}(\mathfrak{p})},$$ where $\kappa_1>0$ is another absolute constant. 
So by Mertens' theorem (over general number fields), we obtain that for any $2\leq w_1\leq w_2$,
$$	\prod_{\substack{w_1\leq \mathbf{N}(\mathfrak{p})\leq w_2}}\left(1-\frac{\omega(\mathfrak{p})}{\mathbf{N}(\mathfrak{p})}\right)^{-1}\leq \prod_{\substack{w_1\leq \mathbf{N}(\mathfrak{p})\leq w_2}}\left(1-\frac{\kappa_1}{\mathbf{N}(\mathfrak{p})}\right)^{-1}\leq \kappa_2\left(\frac{\log w_2}{\log w_1}\right)^{\kappa_1},$$ for an appropriate absolute constant $\kappa_2>1$.

We are now in a position to apply Theorem \ref{thm:sieve}.  According to \eqref{eq:Svol} (cf. also \cite[Theorems 1.4 \& 9.1]{Benoist-Oh}), there exist $a\in\BQ_{>0},b\in\BZ_{\geq 0}$ such that $$\boldsymbol{X}(N)\asymp N^a(\log N)^b.$$ On taking $\CP$ to be the set of prime ideals of $\CO_{S_0}$, and $y=N^{\alpha},z=N^{\beta}$ for sufficiently small $\alpha,\beta>0$, so that, thanks to \eqref{eq:leveldist},
$$\sum_{\substack{\mathfrak{d}\in\CI_{S_0}\\\mathbf{N}(\mathfrak{d})\leq y,\mathfrak{d}\mid \CP(z)}}\left|	\#\CA(N)_{\mathfrak{d}}-\frac{\omega(\mathfrak{d})}{\mathbf{N}(\mathfrak{d})}\boldsymbol{X}(N)\right|=O(\boldsymbol{X}(N)^{1-\delta'}),$$ for a certain $0<\delta'=\delta'(\kappa,\alpha,\beta)<\delta$. On the other hand, Mertens' theorem also yields that $$\prod_{\mathfrak{p}\mid\CP(z)}\left(1-\frac{\omega(\mathfrak{p})}{\mathbf{N}(\mathfrak{p})}\right)\gg (\log z)^{-\kappa_1}\asymp (\log N)^{-\kappa_1}.$$ Then by Theorem \ref{thm:sieve}, imposing $\beta\leq \lambda\alpha$, we finally obtain
\begin{align*}
	\CS(\CA(N),\CP,z)\gg\frac{\boldsymbol{X}(N)}{(\log N)^{\kappa_1}}\gg_\varepsilon N^{a-\varepsilon}.
\end{align*}

We now arrive at the stage of concluding our proof. Let $M>0$ be fixed. Let $N$ be sufficiently large so that $N^\beta>M$. The by the definition of sifting function \eqref{eq:sifting} with respect to the sequence \eqref{eq:seqab}, there exists $g\in \Gamma(\mathfrak{l}_0)\subset \Phi$ such that  whenever $v\notin S_0$ divides $f(g)$ then $\mathbf{N}(\mathfrak{p}_{v})\geq N^\beta>M$.  Moreover, for all $v\in S_0\setminus\{v_0\}$, thanks to our imposed local conditions by $\mathfrak{D},\Phi$, we have $|f(g)|_{v}\ll 1$, while the height condition $\operatorname{Ht}_{v_0}(g)\leq N$ implies $|f(g)|_{v_0}\ll N^{\deg f}$. Now by the product formula, $\prod_{v\in\Omega_{k}}|f(g)|_v=1$. We then have $$\prod_{v\notin S_0} |f(g)|_v\gg |f(g)|_{v_0}^{-1}\gg N^{-\deg f},$$ from which we deduce that whenever $v\notin S_0$ divides $f(g)$ then $\mathbf{N}(\mathfrak{p}_{v})\leq N^{\deg f}$, for $N$ large enough. Hence the number of such $v$ is $\leq \deg f/\beta+1$. So we may take $r_0$ to be the largest integer $\leq  \deg f/\beta+1+\#(S_0\setminus S')$. \qed


\section{A criterion towards \APSA \ for general semisimple groups}\label{se:APSA}
The goal of this section is to prove the following technical result.

\begin{theorem}\label{thm:subtorus}
	Let $G$ be a semisimple simply connected linear algebraic group over a number field $k$.
	Let $D\subset G$ be a closed subset, $T\subset G$ be an anisotropic torus (as a closed subgroup) and $\pi: G\to Y:=G/T$ be the quotient map as a $T$-torsor.
	Assume the following:
	\begin{center}
	    $(*)$ For every $y\in Y$, the fibre $D_y:=D\cap \pi^{-1}(y)$ of $D$ is either empty or of $\dim(D_y)=0$,\\ and the generic fibre $D_{\eta_Y}$ is empty.

	\end{center}
	Let $v_0$ be a place of $k$. If $T(k_{v_0})$ is not compact, then the open subset $U:=G\setminus D$ satisfies strong approximation off $v_0$.
\end{theorem}

\subsection{Consequence on \APSA \ for $3$-dimensional groups and affine quadrics}
Taking Theorem \ref{thm:subtorus} on faith, we obtain the following theorem which generalizes  \cite[Theorem 4.1]{Cao-Huang1}.
\begin{theorem}\label{thm:3dim}
	Let $G$ be a three-dimensional semisimple simply connected linear algebraic group over a number field $k$.
     Let $v_0$ be a place of $k$.
	If $G(k_{v_0})$ is not compact, then $G$ satisfies \APSA \ off $v_0$.
\end{theorem}
\begin{proof}
	The case where $G$ is isotropic is covered by \cite[Theorem 1.4]{Cao-Huang1}. Assuming $G$ is anisotropic and fixing a closed subset $D\subset G$ of codimension at least two in what follows, the first part of \cite[\S4.3.1 Step I]{Cao-Huang1} shows that $G$ is $k$-simple and all its maximal tori are one-dimensional.
	 By \cite[Theorem 3.1]{PR}, $G_{k_{v_0}}$ is isotropic, so it contains an one-dimensional isotropic torus $T_{0}\sbt G_{k_{v_0}}$ over $k_{v_0}$ and $T_0$ is maximal. By \cite[\S 7.1, Corollary 3]{PR}, there exists a maximal torus $T\subset G$ such that $T_{k_{v_0}}$ is conjugate to $T_{0}$ over $k_{v_0}$. 
	Thus $T(k_{v_0})$ is not compact. 
	Since $G$ is anisotropic, $T$ is also anisotropic.
	The second part of \cite[\S4.3.1 Step I]{Cao-Huang1} shows that, upon replacing $T$ by a certain $k$-conjugation if necessary, the quotient map $\pi: G\to Y:=G/T$ meets the requirement $(*)$ of Theorem \ref{thm:subtorus}. 
	From Theorem \ref{thm:subtorus} we deduce strong approximation off $v_0$ for the open set $G\setminus D$. \end{proof}

Granting Theorem \ref{thm:3dim}, the \APSA \ off $v_0$ property for the affine quadric $X$ \eqref{eq:affquadric} follows \emph{mutatis mutandis} on appealing to the arguments in \cite[\S5.1, \S5.2]{Cao-Huang1}. 
Indeed, for a brief sketch, by \cite[Theorem 1.3]{CLX19}, it suffices to establish \APSA\ off $v_0$ for the spin group $\operatorname{Spin}(q)$ and apply \cite[Theorem 2.1]{Cao-Huang1}. For this we need the following. 
\begin{lemma}
    Let $q$ be a non-degenerate quadratic form in $n$-variables with coefficients in $k$.
    Assume that the group $\operatorname{Spin}(q)$ is not compact over $v_0$. Then for any $2\leq n'<n$, it contains a spin subgroup of a non-degenerate quadratic form in $n'$-variables which is not compact over $v_0$.
\end{lemma}
\begin{proof}
Let us consider the  Grassmannian variety $\operatorname{Gr}(n',n)$. For any field extension $K/k$ and $L\in \operatorname{Gr}(n',n)(K)$, viewed as a vector subspace $L\subset K^n$ of dimension $n'$, we denote the restriction of $q$ (over $K$) to $L$ by $q_L$. The non-degeneracy of such a $q_L$ specifies an Zariski-open set $U\subset\operatorname{Gr}(n',n)$ (this is equivalent to $\det(q_L)\neq 0$), which is non-empty because $q$ is non-degenerate.
Let us define $I_{v_0}= \{L\in U(k_{v_0}):q_L ~\text{is isotropic}\}$. We want to show that $I_{v_0}$ is non-empty and open-closed in $v_0$-adic topology. The open-closedness follows from \cite[Theorem 63:20]{OMeara}. To see that $I_{v_0}\neq\varnothing$, as $q$ is assumed to be $k_{v_0}$-isotropic, we can find a basis $\{e_1,\cdots,e_n\}$ of $k_{v_0}^n$ over which $q(x_1,\cdots,x_n)=x_1x_2+a_3x_3^2+\cdots+a_nx_n^2$, where $a_3,\cdots,a_n\in k_{v_0}^\times$. Then $L=\mathrm{Vect}_{k_{v_0}}(e_1,\cdots,e_{n'})\in I_{v_0}$. 
Now using the weak approximation property  at the place $v_0$ for the open subset $U$ of the rational variety $\operatorname{Gr}(n',n)$, we conclude that $I_{v_0}$ contains a $k$-rational vector subspace $L\subset k^n$ such that $q_L$ is isotropic over $k_{v_0}$.
\end{proof}
On taking $n'=3$, we conclude that the spin group $\operatorname{Spin}(q)$ contains a three-dimensional subgroup which is not compact over $v_0$.

\subsection{Proof of Theorem \ref{thm:subtorus}}
Our line of attack follows basically the \cite[\S4]{Cao-Huang1}, and we refer to \cite[\S4.1]{Cao-Huang1} for an overview. We choose to only outline the steps that work without modification, and give full details when extra inputs are necessary. 

We embed $G\sbt \GL_{n,k}$ and let $\CG$ be the Zariski closure of $G$ in $\GL_{n,\CO}$, $\CT$ (resp.~$\CD$) be the Zariski closure of $T$ (resp.~$D$) in $\CG$, and $\CU:=\CG\setminus \CD$. 
To prove that $U$ satisfies strong approximation off $v_0$, it suffices to fix a finite set of places $S\sbt \Omega_k$ containing $\infty_k\cup \{v_0\}$ and show that for $W\sbt U(\RA_k^{v_0})$ any open subset of form \begin{equation}\label{eq:W}
	W=\prod_{v\in S\setminus \{v_0\}}W_v\times \prod_{v\notin  S}\CU(\CO_v),
\end{equation}
where $W_v\sbt U(k_v)$ is a non-empty open subset for each $v\in S\setminus \{v_0\}$ and $W_v$ is compact if $v\notin \infty_k$, we have
\begin{equation}\label{eq:WU}
	W\cap U(k)\neq\varnothing.
\end{equation}

The image $\pi(D)\subset Y$ is a constructible set. Since the generic fibre $D_{\eta_Y}$ is empty, we have $\overline{\pi(D)}\subsetneq Y $, whence $\dim(\overline{\pi(D)})\leq \dim(Y)-1$.
We fix a non-constant regular function $F\in k[Y]$ such that $$\overline{\pi(D)}\subset E:=(F=0)\subset Y.$$
Let $$V:=Y\setminus E.$$ We have $\pi^{-1}(V)\sbt U$.
Then based on the assumption $(*)$, \cite[\S4.3.2 Step II]{Cao-Huang1} readily provides
\begin{lemma}\label{le:step2}
	Upon enlarging $S$, the following hold.
	\begin{enumerate}
		\item The quotient scheme $\CY:=\CG/\CT$ exists over $\CO_S$. Let $\CE$ be the Zariski closure of $E$ in $\CY$ and let $\CV:=\CY\setminus \CE$,  so that we have the following commutative diagram:
		\[\xymatrix{\pi^{-1}(\CV)\ar@{^{(}->}[r]\ar[d] &\CU\ar@{^{(}->}[r]\ar[rd]_{\pi|_{\CU}} & \CG\ar[d]^{\pi} & \CD\ar[d]\ar@{_{(}->}[l]\\
			\CV\ar@{^{(}->}[rr]&&\CY& \CE;\ar@{_{(}->}[l]
		}\]
		\item The regular function $F$ extends to  $\CO_{S}[\CY]$;
		\item For all $v\notin S\cup\infty_k$, the map $\pi|_{\CU}: \CU(\CO_v)\to \CY(\CO_v)$ is surjective, and $\CV(\CO_v)\neq\varnothing$;
		\item For all $y\in\CY_{\CO_S}$, the fibre $\CD_y:=\CD\cap\pi^{-1}(y)$ of $\CD$ is either empty or has dimension $0$, and \begin{equation}\label{eq:L}
			L:=\sup_{y\in \CY_{\CO_S}:\dim (\CD_y)=0}\deg (\CD_y)<\infty.
		\end{equation} 
	\end{enumerate}
\end{lemma}

Our next manipulation is similar to \cite[\S4.3.3 Step III]{Cao-Huang1}.
Let us consider
\begin{equation}\label{eq:W1}
	W_1:= \prod_{v\in S\setminus \{v_0\}}W_v\times \prod_{v\notin S}\CG(\CO_v)\subset G(\RA_{k}^{v_0}),
\end{equation} where $W_v$ are components of the set $W$ \eqref{eq:W}.
For each $v\in S\setminus (\infty_k\cup \{v_0\})$, by \cite[Lemma 4.2]{Cao-Huang1}, there exists an open compact subgroup $\Phi_v\sbt G(k_v)$ such that $\Phi_v\cdot W_v=W_v$. 
For each $v\in \infty_k\setminus \{v_0\}$, we consider the multiplication map
\begin{align*}
	m_v: G(k_v)\times G(k_v)\times G(k_v)&\to G(k_v)\\ (a,b,c)&\mapsto a\cdot b\cdot c.
\end{align*} Then for any $x_v\in W_v$, we obtain an open neighbourhood $m_v^{-1}(W_v)$ of $1_G\times 1_G\times x_v$. Hence there exist open neighbourhoods $\Phi_v$ of $1_G$ and $ W_{2,v}$ of $x_v$ such that $\Phi_v\times \Phi_v\times W_{2,v}\subset m_v^{-1}(W_v)$ (i.e. $\Phi_v\cdot \Phi_v\cdot W_{2,v}\subset W_v$).

Having equipped with all these, let us consider
\begin{equation}\label{eq:W2}
	W_2:= \prod_{v\in \infty_k\setminus \{v_0\}} W_{2,v}\times   \prod_{v\in S\setminus (\infty_k\cup \{v_0\})} W_v \times  \prod_{v\notin S}\CG(\CO_v)\subset G(\RA_{k}^{v_0}),
\end{equation}
$$\Phi_G:=\prod_{v\in S\setminus \{v_0\}}\Phi_v\times \prod_{v\notin S}\CG(\CO_v)\subset G(\RA_{k}^{v_0}),$$  and  
\begin{equation}\label{eq:PhiT}
	\Phi_T:=\prod_{v\in S\setminus \{v_0\}}(\Phi_v\cap T(k_v))\times \prod_{v\notin S} \CT(\CO_v) \subset  T(\RA_{k}^{v_0}).
\end{equation}
We then have 
\begin{equation}\label{eq:W1fix}
	1_T\in \Phi_T,\quad \Phi_G\cdot \Phi_G\cdot  W_2\subset W_1 \quad\text{and}\quad \Phi_T\cdot \Phi_G \cdot W_2\subset W_1.
\end{equation} 
We need the following lemma which generalizes \cite[\S4.5 Corollary 1]{PR}.
\begin{lemma}\label{lem:3dim-torus}
Let $\BT$ be an anisotropic torus over $k$ with identity element $1_{\BT}$. Let $v_1$ be a place such that $\BT(k_{v_1})$ is not compact. Then, for any open subset $1_{\BT}\in \Phi\subset \BT(\RA_k^{v_1})$, the set $\BT(k)\cap \Phi$ is infinite.
\end{lemma}
\begin{proof}
We consider the continuous map: \begin{align*}
	\phi: \BT(\RA_k^{v_1})\times \BT(\RA_k^{v_1})&\to \BT(\RA_k^{v_1})\\  (a,b)&\mapsto a^{-1}\cdot b.
\end{align*}
Then $\phi(1_{\BT},1_{\BT})=1_{\BT}$, hence there exists an open neighbourhood $\Phi_1\subset \BT(\RA_k^{v_1})$ of $1_{\BT}$ such that $\phi(\Phi_1,\Phi_1)\subset \Phi$.
We consider the quotient $\BT(\RA_k)\to \BT(\RA_k)/\BT(k)$.
By \cite[Theorem 5.6]{PR}, $\BT(\RA_k)/\BT(k)$ has finite invariant volume. 
Since $\BT(k_{v_1})$ is not compact, $\BT(k_{v_1})\times \Phi_1\subset \BT(\RA_k)$ has infinite invariant volume, hence there are infinitely many $t\in T(k)$ such that $ (\BT(k_{v_1})\times \Phi_1) \cdot t \cap (\BT(k_{v_1})\times \Phi_1)\neq \varnothing$.
It follows that any such $t$ belongs to $ \Phi$ since $\phi(\Phi_1,\Phi_1)\subset \Phi$. This proves infinitude of $\#(\BT(k)\cap \Phi)$.
\end{proof}
By assumption, $G$ is anisotropic and $T(k_{v_0})$ is not compact. Lemma \ref{lem:3dim-torus} implies that $T(k)\cap \Phi_T$ is infinite (and countable). We can then list its elements as a sequence
\begin{equation}\label{eq:Q}
	T(k)\cap \Phi_T=(Q_i)_{i\in \BN}.
\end{equation} 
For all $v\in \Omega_k\setminus S$, let us write $$\operatorname{Mod}_v=\operatorname{Mod}_{v,1}:\CG(\CO_{S})\to \CG(\BF_{\mathfrak{p}_{\nu}})$$ for the reduction modulo $v$ map. We recall from construction \eqref{eq:PhiT} that $T(k)\cap \Phi_T\subset \CT(\CO_{S})$. 
\begin{lemma}\label{le:step3vi}
	Let $r\in\BN_{\geq 1}$ be fixed. Then there exists $M_r>0$ such that, for any place $v\in \Omega_{k}\setminus (S\cup\infty_{k})$ with $\#\BF_{\mathfrak{p}_{\nu}}>M_r$, the restriction of the map $\operatorname{Mod}_v$ to the finite subsequence $(Q_i)_{0\leq i\leq rL}$ (where $L$ is defined by \eqref{eq:L}) is injective. 
\end{lemma}
\begin{proof}
	The proof is a minor modification of \cite[\S4.3.2 (vi)]{Cao-Huang1}. Indeed, 
	the set of places
	$$B:=\{v\in \Omega_{k}\setminus (S\cup\infty_{k}):\exists 0\leq i<j\leq rL, \operatorname{Mod}_v(Q_i)= \operatorname{Mod}_v(Q_j)\}$$ is clearly finite. To conclude it suffices to take $M_r:=\max_{v\in B}(\#\BF_{\mathfrak{p}_{\nu}})$.
\end{proof}

Next, making use of strong approximation off $v_0$ for $G$ regarding the adelic set $W_2$ \eqref{eq:W2}, we may fix $P_0\in G(k)\cap W_2$ and define the regular function $F_0\in k[G]$ by 
\begin{equation}\label{eq:F0}
	F_0(g):=(F\circ \pi) (g\cdot P_0).
\end{equation}
Note that $F_0$ extends to $\CO_{S}[\CG]$ by Lemma \ref{le:step2} (2). For any $x\in\CO_{S}$, we introduce the notation $$\Omega_S(x):=\{v\in \Omega_{k}\setminus (S\cup\infty_{k}):v\mid x\}.$$
Our goal now is to establish the sieving result.
\begin{lemma}\label{le:stepIV}
	There exist $r_0\in\BN$ and an element $g_0\in \Phi_G\cap \CG(\CO_{S})$ such that $\#\Omega_S(F_0(g_0))\leq r_0$ and for any $v\in \Omega_S(F_0(g_0))$, we have $\# \BF_{\mathfrak{p}_{\nu}}> M_{r_0}$, where the constant $M_{r_0}$ is given in Lemma \ref{le:step3vi}.
\end{lemma}
\begin{proof}
	This is direct analogue of \cite[\S4.3.4 (vii)]{Cao-Huang1}. Following the same argument in \cite[\S4.3.4 Step IV]{Cao-Huang1}, thanks to Lemma \ref{le:step2} (3), for all $v\in \Omega_{k}\setminus (S\cup\infty_{k})$ we may pick $P_v\in\CU(\CO_v)$ such that $\pi(P_v)\in \CV(\CO_v)$. Then the element $g_v:=P_v\cdot P_0^{-1}$ satisfies $F_0(g_v)=F(\pi(P_v))\in \CO_v^\times$, which implies $F_0(\CG(\CO_v))\cap \CO_v^{\times}\neq \varnothing$.
	Now for any $v\in \infty_k\setminus \{v_0\}$, we may fix a compact region $\mathfrak{D}_v$ with non-empty interior and piecewise smooth boundary such that $\mathfrak{D}_v\subset \Phi_v$, and let 
	$$\mathfrak{D}:=\prod_{v\in \infty_k\setminus \{v_0\}}\mathfrak{D}_v\quad \text{and}\quad  \Phi^{\infty}_G:=\prod_{v\in S\setminus (\infty_k\cup \{v_0\})}\Phi_v\times \prod_{v\notin S}\CG(\CO_v).$$ Note that  $\mathfrak{D}\times \Phi^{\infty}_G\subset \Phi_G$.
	Then the hypotheses of Theorem \ref{thm:almostprime} are all satisfied for $(\CG,\mathfrak{D},\Phi^{\infty}_G,f,\infty_k\cup \{v_0\},S)$ 
	(where the $S$ in Theorem \ref{thm:almostprime} is $\infty_k\cup \{v_0\}$ here and the $S'$ in Theorem \ref{thm:almostprime} is the $S$ here). \end{proof}
	
	We now arrive at the final stage of the proof and we adapt the argument of \cite[\S4.3.5 Step V]{Cao-Huang1}. Based on the sequence \eqref{eq:Q} and the pair $(r_0,g_0)$ obtained in Lemma \ref{le:stepIV}, we put $P_1:=g_0\cdot P_0$ and define the new finite sequence 
	$$\Theta:=(Q_i\cdot P_1)_{0\leq i\leq r_0 L}\subset \CG(\CO_{S}),$$ where $L$ is given by \eqref{eq:L}. Then $\Theta\subset G(k)\cap W_1$ thanks to \eqref{eq:W1fix}. Note that $\pi(Q_i\cdot P_1)=\pi(P_1)$ for all $i$ as the fibration $\pi$ is a $T$-torsor, and we write $y_1\in\CY(\CO_{S})$ for their common image, so that $$(F\circ \pi)(Q_i\cdot P_1)=F_0(g_0)=F(y_1),\quad 0\leq i\leq r_0 L.$$ We now consider the subsequence
	$$\Theta_{\operatorname{bad}}:=\{Q_i\cdot P_1,0\leq i\leq r_0 L: \exists v\in\Omega_S(F(y_1)),\operatorname{Mod}_v(Q_i\cdot P_1)\in\CD_{y_1}(\BF_{\mathfrak{p}_{\nu}})\}\subset\Theta.$$
		By Lemma \ref{le:step2}, for all $v\in \Omega_{k}\setminus (S\cup\infty_{k})$, we have $\#\CD_{y_1}(\BF_{\mathfrak{p}_{\nu}})\leq \deg\CD_{y_1}\leq L$; By Lemma \ref{le:stepIV}, we have $\#\Omega_S(F(y_1))\leq r_0$ and  for all  $v\in\Omega_S(F(y_1))$, $\#\BF_{\mathfrak{p}_{\nu}}>M_{r_0}$; Moreover, by Lemma \ref{le:step3vi},  for all  $v\in\Omega_S(F(y_1))$, the map $\operatorname{Mod}_v$ restricted to the sequence $\Theta$ stays injective. 
	Gathering together all these ingredients we can estimate
	$$\#\Theta_{\operatorname{bad}}\leq \#\Omega_S(F(y_1))\sup_{v\in\Omega_S(F(y_1))}\#\CD_{y_1}(\BF_{\mathfrak{p}_{\nu}})\leq r_0L<\#\Theta.$$
	Therefore, picking any $P_2\in \Theta\setminus\Theta_{\operatorname{bad}}$, we have that for all $v\in \Omega_{k}\setminus (S\cup\infty_{k})$, $\operatorname{Mod}_v(P_2)\in \CU_{y_1}(\BF_{\mathfrak{p}_{\nu}})$, and hence $$P_2\in W\cap G(k)=W\cap U(k),$$ achieving \eqref{eq:WU}. This finishes the proof of Theorem \ref{thm:subtorus}. \qed

\section{The geometric sieve for $p$-integral points on affine quadrics}\label{se:geomsieve}
Let us fix an integral model of our affine quadric $X$: $$\CX:(q=m)\subset \BA^n_{\BZ},$$ where $q$ is a non-degenerate integral quadratic form in $n$-variables, and $m$ is a non-zero integer. 
Assume throughout this section that either $n\geq 4$; or $n=3,-m\det q\neq\square$ and $q$ is $\BQ$-anisotropic.
Let us fix a prime number $p_0$, and define the $p_0$-adic height function on $\BA_{\BQ}^{n}$:
$$\|(x_1,\cdots,x_n)\|_{p_0}:=\max_{1\leq i\leq n}|x_i|_{p_0}.$$

Now let $f,g\in\BQ[\BA^n]$ be two coprime polynomials such that the closed subset $Z:=X\cap(f=g=0)\subset\BA^n$ has codimension two in $X$. We may assume that $f,g$ has integer coefficients. Let $\CZ$ be the Zariski closure of $Z$ in $\CX$. 
We fix $\CF\subset X(\BR)$ a compact neighbourhood with non-empty interior and piecewise smooth boundary. 
For $h\in\BN$, we write
$$\BB_{p_0}(h):=\{\bx\in X(\BQ_{p_0}):\|\bx\|_{p_0}\leq p_0^h\}.$$
With the effective equidistribution result (Theorem \ref{thm:equidistcongruence}) at hand, according the strategy towards \APHL\ developed in \cite[Theorem 3.1]{Cao-Huang2}, everything boils down to the following result on the geometric sieve analogous to \cite[Theorem 1.6]{Cao-Huang2}.

\begin{theorem}\label{thm:geosieve}
 Under the setting above, uniformly for $M>p_0,h\to\infty$, we have
	\begin{align*}
		&\#\left\{\bx\in \CX(\BZ\left[p_0^{-1}\right])\cap\CF\cap \BB_{p_0}(h):\exists p>M,\bx\bmod p\in \CZ(\BF_p) \right\}\\ &\ll p_0^{h(n-2)}\left(\frac{1}{\sqrt{\log (p_0^h)}}+\frac{1}{M\log M}\right).
	\end{align*}
\end{theorem}
The rest of this section is devoted to the proof of Theorem \ref{thm:geosieve}, for which we proceed along a similar way as in \cite[\S4]{Cao-Huang2} and only sketch the main steps.

\subsection*{Proof of Theorem \ref{thm:geosieve}}
For every $p_0<N_1<N_2\leq \infty$, we want to estimate
the truncated counting function $$\BV(N_1,N_2):=\#\{\bx\in \CX(\BZ\left[p_0^{-1}\right])\cap \CF\cap \BB_{p_0}(h):\exists N_1<p\leq N_2,\bx\bmod p \in\CZ(\BF_p)\}.$$ 
The count in Theorem \ref{thm:geosieve} equals $\BV(M,\infty)$. For every prime $p>p_0$ and $\blambda_p\in\CX(\BF_p)$ we consider the congruence neighbourhood
$$\CE^{\CX}_{p_0}(p;\blambda_p):=\{\bx_p\in\CX(\BZ_p):\bx_p\equiv \blambda_p\bmod p\}\times\prod_{p'\neq p, p_0}\CX(\BZ_{p'}).$$ Then $\bx\in \CX(\BZ\left[p_0^{-1}\right]),\bx\bmod p \in\CZ(\BF_p)$ if and only if there exists $\blambda_p\in\CZ(\BF_p)$ such that $\bx\in X(\BQ)\cap \CE^{\CX}_{p_0}(p;\blambda_p)$.
Our first alternative bound is similar to \cite[Theoreme 4.1]{Cao-Huang2}.
\begin{proposition}\label{prop:Vbd1}
	There exist $\delta>0,\kappa'>0$ such that uniformly for all $p_0< N_1<N_2<\infty$, \begin{equation}\label{eq:bdV1}
		\BV(N_1,N_2)\ll \frac{p_0^{h(n-2)}}{N_1\log N_1}+N_2^{\kappa'}p_0^{h(n-2)(1-\delta)}.
	\end{equation}
\end{proposition}
\begin{proof}[Proof of Proposition \ref{prop:Vbd1}]
	By Theorem \ref{thm:equidistcongruence} about the effective equidistribution of $p_0$-integral points, we obtain that uniformly for such $\blambda_p$,
\begin{align*}
	\#\left(X(\BQ)\cap \CF\times \BB_{p_0}(h)\times \CE^{\CX}_{p_0}(p;\blambda_p)\right)&=\int_{\CF\times \BB_{p_0}(h)\times \CE^{\CX}_{p_0}(p;\blambda_p)}\boldsymbol{\delta}\operatorname{d}\rmm_X+O(p^{\kappa+\frac{n(n-1)}{2}}\rmm_{X_{p_0}}(\BB_{p_0}(h))^{1-\delta})\\ &\ll \frac{\rmm_{X_{p_0}}(\BB_{p_0}(h))}{p^{n-1}}+p^{\kappa+\frac{n(n-1)}{2}}\rmm_{X_{p_0}}(\BB_{p_0}(h))^{1-\delta}.
\end{align*}

By the Lang--Weil estimate applied to $\CZ$ we have \begin{equation}\label{eq:LWforZ}
	\#\CZ(\BF_p)\ll p^{n-3}
\end{equation} uniformly for all $p$. Appealing to \eqref{eq:affquavol} and summing over all residues $\blambda_p\in\CZ(\BF_p)$, we thus obtain the bound \eqref{eq:bdV1}, with $\kappa'=\kappa+\frac{n(n-1)}{2}+n-2$.
\end{proof}
 
 Our next alternative bound is analogous to \cite[Theorem 4.2]{Browning-Gorodnik}.
Before stating it, we make the following elementary observation. Let $\bx\in (\BZ[p_0^{-1}])^n$. Then the height condition $\|\bx\|_{p_0}\leq p_0^h$ is equivalent to $\by:=p_0^h \bx\in\BZ^n$.  Hence under this height condition, $\bx\in \CX(\BZ[p_0^{-1}])$ if and only if $\by$ is an $\BZ$-point of the affine quadric $$\CX_{h}:q=p_0^{2h}m,$$ and the condition $\bx\in\CE^{\CX}_{p_0}(p;\blambda_p)$ is equivalent to $\by\equiv p_0^{h}\blambda_p\bmod p$ in $\CX_{h}(\BF_p)$.
 We note that $\bx\in\CF$ implies that the archimedean norm $|\bx|$ of $\bx$ is bounded (depending on the fixed $\CF$), which in term implies that $|\by|\ll p_0^h$.  
 
\begin{proposition}\label{prop:Vbd2}
	For any $0<\alpha<1$, there exists $0<\delta(\alpha)<1$ such that  we have 
$$\BV(p_0^{\alpha h},p_0^{h})\ll \frac{p_0^{h(n-2)}}{\log (p_0^h)}+p_0^{h(n-2)(1-\delta(\alpha))}$$
\end{proposition}

\begin{proof}[Proof of Proposition \ref{prop:Vbd2}]
		We shall prove that, uniformly for any prime $1\ll p\leq p_0^{h}$ different from $p_0$, any $\blambda_p\in\CX(\BF_p)$, we have
	\begin{equation}\label{eq:bdp1}
		\#\left(X(\BQ)\cap \CF\times \BB_{p_0}(h)\times \CE^{\CX}_{p_0}(p;\blambda_p)\right)\ll_\varepsilon \left(\frac{p_0^h}{p}\right)^{n-3+\varepsilon}\left(1+\frac{p_0^h}{p^{\frac{n+1}{n}}}\right).
	\end{equation}
Taking \eqref{eq:bdp1} for granted, it is straightforward to deduce the desired bound using the Lang--Weil estimate for $\CZ$ \eqref{eq:LWforZ}: 	\begin{align*}
	\BV(p_0^{\alpha h},p_0^{h})&\leq p_0^{h(n-3)} \sum_{p_0^{\alpha h}<p\leq p_0^{h}}\left(\left(\frac{p_0^h}{p}\right)^\varepsilon+\frac{p_0^{h(n-2+\varepsilon)}}{p^{1+\frac{1}{n}+\varepsilon}}\right)\ll \frac{p_0^{h(n-2)}}{\log (p_0^h)}+p_0^{h(n-2)(1-\frac{\alpha}{n}+\varepsilon)}.
\end{align*}

Based on the observation above, the estimation of \eqref{eq:bdp1} boils then down to 
$$\#\{\by\in\CX_{h}(\BZ):|\by|\ll p_0^h,\by\equiv p_0^h\blambda_p\bmod p\}.$$
Either the set to be counted is empty, or it contains  an integral point $\blambda_h\in \CX_h(\BZ)$ satisfying $\blambda_h\equiv p_0^h\blambda_p\bmod p$. Executing the change of variables $\by=\blambda_h+p\bz$, the new variable $\bz$ satisfies
$$|\bz|\ll\frac{p_0^h}{p},\quad \bz\cdot\nabla q(\blambda_h)+pq(\bz)=0.$$ In particular $\bz$ satisfies $\bz\cdot\nabla q(\blambda_h)\equiv \bz\cdot \nabla q(p_0^h\blambda_p)\equiv 0\bmod p$, which is also equivalent to the lattice condition $\bz\cdot\nabla q(\blambda_p)\equiv 0\bmod p$.
The argument in \cite[p.1076--1077]{Browning-Gorodnik} shows that the determinant of this lattice is  $\asymp p$ for $1\ll p\neq p_0$. 
The uniform point counting technique as in \cite[p. 1076--1078]{Browning-Gorodnik} (for $n=4$) and in \cite[\S4.2.2]{Cao-Huang2} (for $n=3,-m\det q\neq\square$ and $q$ anisotropic) can be accordingly adapted and yields \eqref{eq:bdp1}. The details will not be repeated.
\end{proof}

Our final bound deals with the primes with $p\geq p_0^h$, analogous to \cite[Theorem 4.7]{Cao-Huang2}.
\begin{proposition}\label{prop:Vbd3}
	We have
	\begin{equation}\label{eq:bV3}
		\BV(p_0^{h},\infty)=O\left(\frac{p_0^{h(n-2)}}{(\log p_0^h)^\frac{1}{2}}\right).
	\end{equation}
\end{proposition}
\begin{proof}[Proof of Proposition \ref{prop:Vbd3}]
We may diagonalise $q$ and switch signs so that $q(\bx)=\sum_{i=1}^n a_ix_i^2$ with $a_n=1,a_{n-1}=a>0$. Moreover, we can assume that $f\in\BZ[x_1,\cdots,x_{n-1}],g\in \BZ[x_1,\cdots,x_{n-2}]$ (cf. \cite[p. 38]{Cao-Huang2}). 
 Let the integral polynomials $f_h\in\BZ[x_1,\cdots,x_{n-1}],g_h\in \BZ[x_1,\cdots,x_{n-2}]$ be such that $$f_h(\bx)=p_0^{h\deg f}f\left(\frac{\bx}{p_0^h}\right),\quad g_h(\bx)=p_0^{h\deg g}g\left(\frac{\bx}{p_0^h}\right).$$ Then whenever $|\by|\ll p_0^h$, $$|f_h(\by)|\ll p_0^{h\deg f},\quad |g_h(\by)|\ll p_0^{h\deg g}.$$
 
 Let $\pr:\BA^n\to\BA^{n-2}$ be the projection onto the first $(n-2)$ coordinates.
 Let $$\scrZ:=\{\scrY\in\pr(\CZ):\operatorname{codim}_{\BA^2_{\scrY}}(\CZ_{\scrY})\leq 1\},$$ 
 and $$\CB_1:=\left\{\bY\in\BZ^{n-2}:\text{either }g_h(\bY)=0\text{ or } \exists p\geq p_0^h, \bY\bmod p\in \scrZ\right\},$$ $$\CB_2:=\left\{\bY\in\BZ^{n-2}:g_h(\bY)\neq 0~\text{and for all}~p\geq p_0^h, f_h(\bY,x_{n-1})\bmod p \in\BF_p[x_{n-1}]~\text{is non-zero}\right\}.$$ Rerunning the argument of \cite[Proof of Theorem 4.7]{Cao-Huang2}, $\BV(p_0^{h},\infty)$ can be bounded from above by two parts: 
 $$\BV_1:=\sum_{\substack{\bY\in\CB_1,|\bY|\ll p_0^h}}\scrH_1(\bY),\quad \BV_2:=\sum_{\substack{\bY\in\CB_2,|\bY|\ll p_0^h\\\exists u,v\in\BZ,u^2+av^2=mp_0^{2h}-\sum_{i=1}^{n-2}a_iy_i^2}} \scrH_2(\bY),$$
where for $\bY=(y_1,\cdots,y_{n-2})\in\BZ^{n-2}$, we let $$\scrH_1(\bY):=\#\left\{(u,v)\in \BZ^2:u^2+av^2=mp_0^{2h}-\sum_{i=1}^{n-2}a_iy_i^2\right\},$$  $$\scrH_2(\bY):=\sum_{\substack{p:p\geq p_0^h\\p\mid g_h(\bY)}}\sum_{\substack{y\in\BZ:|y|\ll p_0^h\\ p\mid f_h(\bY,y)}}\#\left\{z\in\BZ:z^2=mp_0^{2h}-ay^2-\sum_{i=1}^{n-2}a_iy_i^2\right\}.$$ 

As $a>0$, it is easy to see that uniformly for all $\bY$ with $|\bY|\ll p_0^h$, we have $\scrH_1(\bY)\ll_\varepsilon p_0^{\varepsilon h}$. To deal with $\bY\in\CB_1$, we observe that for every prime $p$ the fibre $\CZ_{\bY}\bmod p\subset\BA_{\BF_p}^2$ is defined by the equations $$mp_0^{2h}-\sum_{i=1}^{n-2}a_iy_i^2- ax_{n-1}^2+x_n^2\equiv f_h(\bY,x_{n-1})\equiv 0\mod p.$$ If $p$ is large enough, the first polynomial is always non-zero, and hence the condition $\bY\bmod p\in \scrZ$ implies that $p\mid g_h(\bY)$ and the polynomial $f_h(\bY,x_{n-1})\bmod p$ in $x_{n-1}$ is zero, i.e., $p$ divides all its coefficients, say $f_{1,h}(\bY),\cdots,f_{\deg f,h}(\bY)$.
Since the set $\scrZ$ has codimension at least two in $\BA^{n-2}_{\BZ}$ by \cite[(67)]{Cao-Huang2}, we now invoke the uniform version of Ekedahl's sieve \cite{Ekedahl} obtained by work of Browning--Heath-Brown \cite[Lemmas 2.1 \& 2.2]{Browning-HB}, yielding that
 $$\#\left\{\bY\in\BZ^{n-2}:|\bY|\ll p_0^h,g_h(\bY)=0\right\}\ll p_0^{(n-3)h},$$ and
 \begin{align*}
 	&\#\left\{\bY\in\BZ^{n-2}:|\bY|\ll p_0^h,\exists p\geq p_0^h, p\mid \gcd(g_h(\bY),f_{1,h}(\bY),\cdots,f_{\deg f,h}(\bY))\right\} \\ \ll &\frac{p_0^{(n-2)h}\log(p_0^{h(\max(\deg f,\deg g)+1)})}{p_0^h}\ll_\varepsilon p_0^{(n-3+\varepsilon)h}.
 \end{align*}
 Therefore we obtain $$\BV_1\ll_\varepsilon p_0^{(n-3+\varepsilon)h}.$$
 
 As for $\BV_2$, since $f_h(\bY,x_{n-1})\bmod p$ is a non-zero polynomial in $x_{n-1}$ for all $\bY\in\CB_2$ and $p\geq p_0^h$, the argument of \cite[p. 41]{Cao-Huang2} shows that $\scrH_2(\bY)=O(1)$ uniformly for all $h$, whence $\BV_2\ll \scrC$, where $$\scrC:=\#\left\{\bY=(y_1,\cdots,y_{n-2})\in\BZ^{n-2}:|\bY|\ll p_0^h,\exists u,v\in\BZ,u^2+av^2=mp_0^{2h}-\sum_{i=1}^{n-2}a_iy_i^2\right\}.$$
We now appeal to the half-dimensional sieve \cite[Theorem 1.7]{Cao-Huang2}, the major difference being that the constant term $mp_0^{2h}$ varies. As the sieving process only involves reductions modulo sufficiently large primes and the Lang--Weil estimate, the argument in \cite[\S5]{Cao-Huang2}  goes without change and leads to 
$$\scrC\ll \frac{p_0^{h(n-2)}}{(\log (p_0^h))^\frac{1}{2}},$$
 whence $$\BV_2\ll \frac{p_0^{h(n-2)}}{(\log (p_0^h))^\frac{1}{2}}.$$ We finally conclude \eqref{eq:bV3}. \end{proof}

To finish the proof of Theorem \ref{thm:geosieve}, it remains to combine the various alternative bounds obtained in Propositions \ref{prop:Vbd1} \ref{prop:Vbd2} \ref{prop:Vbd3} (by choosing $N_1=M$, $N_2=p_0^{\beta h}$ with $\beta\kappa'<\delta$). \qed

\appendix
\section{The Siegel--Weil weight formulae}\label{se:Siegel-Weil}
The goal of this appendix is to state an $S$-integral version of the Siegel--Weil formula for homogeneous spaces with unimodular stabilizers, generalizing work of Borovoi--Rudnick \cite[Theorem 4.2]{Borovoi-Rudnick} and upgrading \cite[Theorem 3.2]{Cao-Huang2}.

Let $k$ be a number field. Let $G$ be a semi-simple simply connected linear algebraic group, and let $X=H\bs G$ be a homogeneous space, where the stabilizer $H$ is connected and has no non-trivial $k$-character. In particular $H$ is unimodular. For $x\in X(k)$ we write $H_x$ for the stabilizer of $x$ under $G$. Let $S$ be a finite set of places. We associate compatible Haar measures $\rmm_{G},\rmm_{H},\rmm_{X}$. We write $\rmm_{G_S}$ (resp. $\rmm_{G}^{S}$)  on the corresponding space $G_S$ (resp. $G(\RA_k^S)$) and similarly for $H,X$.

Now we place ourselves in a setting that is more coherent to \S\ref{se:effective}. Assume that $S$ contains all archimedean places $v$ over which $G(k_v)$ is not compact. Note that $\infty_k\setminus S$ consists entirely of places over which $G$ is compact. We shall write $$\left|\rmm_{G_{S^c}}\right|:=\rmm_{G_{\infty_k\setminus S}}(G_{\infty_k\setminus S}),\quad \left|\rmm_{H_{S^c}}\right|:=\rmm_{H_{\infty_k\setminus S}}(H_{\infty_k\setminus S}),\quad \left|\rmm_{X_{S^c}}\right|:=\rmm_{X_{\infty_k\setminus S}}(X_{\infty_k\setminus S}).$$ Let $\mathbf{O}_{S}$ be a $G_S$-orbit in $X_S$ and $B^S$ be a compact-open subset of $X(\RA_k^S)$.  Let $K^S\subset G(\RA_k^S)$ be an open-compact subgroup such that $K^S\cdot B^S=B^S$. Let $$B:=\mathbf{O}_S\times B^S\subset X(\RA_k),\quad K:=G_S\times K^S\subset G(\RA_k).$$ Then $$\Gamma:=G(k)\cap K$$ embeds as a discrete subgroup of $G_S$. 

\begin{theorem}[cf. \cite{Borovoi-Rudnick} Theorem 4.2]\label{thm:BRSiegelWeil}
	With notation above, assuming that $S$ contains at least one place (not necessarily archimedean) over which every $k$-simple factor of $G$ is not compact, we have
	$$\sum_{\scrO: \text{ orbits of }\Gamma \text{ in } X(k)\cap B} \frac{\rmm_{(H_x)_S}(\Gamma\cap (H_x)_S\bs(H_x)_S)}{\rmm_{G_S}(\Gamma\bs G_S)}=\# C(H) \left|\rmm_{X_{S^c}}\right| \rmm_{X}^S(B^S),$$  where we fix an $x\in \scrO$ for each $\scrO$ (clearly $w(\scrO)$ does not depend on the choice of $x$), and where $C(H)$ is the \emph{Kottwitz} invariant (cf. \cite[\S3.4]{Borovoi-Rudnick}).
\end{theorem}
\begin{proof}[Sketch of proof of Theorem \ref{thm:BRSiegelWeil}]
	The proof follows the same lines as in \cite[\S4]{Borovoi-Rudnick}. Note that the assumption on $S$ ensures that $G$ satisfies strong approximation off $S$ by the theorem of Kneser--Platonov (cf. \cite[Theorem 1]{Cao-Huang1}), i.e. $G(k)$ is dense in $G(\RA_k^S)$. Hence $KG(k)=G(\RA_k)$, and the Tamagawa numbers of $G$ has the following factorization
	$$\tau(G):=\rmm_G(G(k)\bs G(\RA_k))=\rmm_G(G(k)\bs KG(k))=\rmm_{G_S}(\Gamma \bs G_S) \left|\rmm_{G_{S^c}}\right| \rmm_{G}^{S}(K^S).$$ A similar factorization for $\rmm_H$ is 
	$$\rmm_{H_x}(H_x(k)\bs (K\cap H_x(\RA_k))H_x(k))= \rmm_{(H_x)_S}(\Gamma\cap (H_x)_S\bs(H_x)_S)\left|\rmm_{H_{S^c}}\right| \rmm_{H_x}^S\left(K^S\cap H_x(\RA_{k}^{S})\right).$$
	Adapting the argument in \cite[\S4.5]{Borovoi-Rudnick}, we can show that $$\frac{\rmm_{(H_x)_S}(\Gamma\cap (H_x)_S\bs(H_x)_S)}{\rmm_{G_S}(\Gamma\bs G_S)}=\frac{\tau(H_x)}{\tau(G)}\frac{\rmm_{H_x}(H_x(k)\bs (K\cap H_x(\RA_k))H_x(k))}{\tau(H_x)}\left|\rmm_{X_{S^c}}\right|\rmm_{X}^S(K^S x).$$
	Moreover, \cite[Proposition 4.4]{Borovoi-Rudnick} shows that 
	$$\sum_{\scrO: \text{ orbits of }\Gamma \text{ in } X(k)\cap B}\frac{\rmm_{H_x}(H_x(k)\bs (K\cap H_x(\RA_k))H_x(k))}{\tau(H_x)}=\#\Sha^1(H),$$ where $\Sha^1(H)$ is the Tate-Shafarevich group of $H$. Combining results on computation of Tamagawa numbers
	$$\tau(G)=1,\quad \tau(H)=\frac{\#C(H)}{\#\Sha^1(H)},$$ and summing over orbits of $K^S$ in $B^S$ complete the proof.
\end{proof}

\section*{Acknowledgments}
We thank Dasheng Wei and Fei Xu for their interests and helpful discussions. Z. Huang was supported by National Key R\&D Program of China No. 2024YFA1014600, and acknowledges the organizers of the Pan Asian Number Theory Conference 2023, where a preliminary version of this work was presented.

\end{document}